\documentclass[12pt]{amsart}

\usepackage[dvips]{graphics}
\usepackage{epsfig}
\usepackage{amssymb}
\usepackage{amsthm}
\usepackage{amscd}
\usepackage[mathscr]{eucal}
\usepackage{enumitem}

\usepackage{fullpage}

\theoremstyle{plain}
\newtheorem{theorem}{Theorem}[section]
\newtheorem{corollary}[theorem]{Corollary}
\newtheorem{lemma}[theorem]{Lemma}

\newtheorem{proposition}[theorem]{Proposition}
\newtheorem{fact}[theorem]{Fact}

\newtheorem{claim}[theorem]{Claim}

\theoremstyle{definition}
\newtheorem{definition}[theorem]{Definition}
\newtheorem{remark}[theorem]{Remark}

\newtheorem{notation}[theorem]{Notation}

\theoremstyle{remark}

\newcommand{\mor}{\textup{Mor}}

\newcommand{\dom}{\operatorname{dom}}

\newcommand{\tp}{\operatorname{tp}}

\newcommand{\acl}{\operatorname{acl}}
\newcommand{\dcl}{\operatorname{dcl}}
\newcommand{\cl}{\operatorname{cl}}
\newcommand{\id}{\operatorname{id}}

\newcommand{\ov}{\overline}
\newcommand{\be}{\begin{enumerate}}
\newcommand{\ee}{\end{enumerate}}

\renewcommand{\phi}{\varphi}

\def\Ind{\setbox0=\hbox{$x$}\kern\wd0\hbox to 0pt{\hss$\mid$\hss}
\lower.9\ht0\hbox to 0pt{\hss$\smile$\hss}\kern\wd0}

\def\Notind{\setbox0=\hbox{$x$}\kern\wd0\hbox to 0pt{\mathchardef
\nn=12854\hss$\nn$\kern1.4\wd0\hss}\hbox to
0pt{\hss$\mid$\hss}\lower.9\ht0 \hbox to
0pt{\hss$\smile$\hss}\kern\wd0}

\def\ind{\mathop{\mathpalette\Ind{}}}

\newbox\noforkbox \newdimen\forklinewidth
\forklinewidth=0.3pt \setbox0\hbox{$\textstyle\smile$}
\setbox1\hbox to \wd0{\hfil\vrule width \forklinewidth depth-2pt
 height 10pt \hfil}
\wd1=0 cm \setbox\noforkbox\hbox{\lower 2pt\box1\lower
2pt\box0\relax}
\def\unionstick{\mathop{\copy\noforkbox}\limits}

\def\nonfork_#1{\unionstick_{\textstyle #1}}

\setbox0\hbox{$\textstyle\smile$} \setbox1\hbox to \wd0{\hfil{\sl
/\/}\hfil} \setbox2\hbox to \wd0{\hfil\vrule height 10pt depth
-2pt width
               \forklinewidth\hfil}
\newbox\doesforkbox
\setbox\doesforkbox\hbox{\lower 2pt\box1 \lower
2pt\box2\lower2pt\box0\relax}
\def\nunionstick{\mathop{\copy\doesforkbox}\limits}

\def\fork_#1{\nunionstick_{\textstyle #1}}

\newcommand{\C}{\operatorname{\mathfrak{C}}}
\renewcommand{\P}{{\mathcal P}}
\newcommand{\Pm}{{\mathcal P}^{-}}

\newcommand{\Aut}{\operatorname{Aut}}

\newcommand{\bd}{\partial}

\def\cA{\mathcal{A}}
\def\cC{\mathcal{C}}

\def\bd{\partial}

\def\supp{\textup{supp}}

\def\ov{\overline}
\def\Star{\textup{Star}}

\title[Amalgamation and polygroupoids]{Type-amalgamation properties and polygroupoids in stable theories}
\author{John Goodrick}
\author{Byunghan Kim}
\author{Alexei Kolesnikov}

\address{Department of Mathematics\\ Universidad de los Andes \\
Bogot\'{a}, Colombia}
\address{Department of Mathematics\\ Yonsei University\\
134 Shinchon-dong, Seodaemun-gu\\
Seoul 120-749, South Korea}
\address{Department of Mathematics\\ Towson University, MD\\
USA}

\email{jr.goodrick427@uniandes.edu.co}
\email{bkim@yonsei.ac.kr}
\email{akolesnikov@towson.edu}

\thanks{The  second author was supported by NRF of Korea grant 2011-0021916, and Samsung
Science Technology Foundation grant BA1301-03.}
\thanks{The third author was partially supported by NSF grant DMS-0901315.}

\begin{document}

\begin{abstract}
We show that in a stable first-order theory, the failure of higher-dimensional type amalgamation can always be witnessed by algebraic structures we call \emph{$n$-ary polygroupoids}. This generalizes a result of Hrushovski in \cite{Hr} that failures of $4$-amalgamation are witnessed by definable groupoids (which correspond to $2$-ary polygroupoids in our terminology).  The $n$-ary polygroupoids are definable in a mild expansion of the language (adding a predicate for a Morley sequence).
\end{abstract}

\maketitle

\section*{Introduction and outline of results}

This paper provides a characterization of the failure of type amalga\-mation properties in a stable first-order theory $T$ in terms of definable algebraic objects. Suppose that $n$ is the smallest integer such that $T$ \textbf{fails} to have $(n+2)$-amalgamation (we explain what this means in the next paragraph). We give a method for constructing an algebraic object $\mathcal{H}$ in $T$ which we call an \emph{$n$-ary polygroupoid} that witnesses this failure of amalgamation and which is definable in a mild expansion of the language of $T$ (adding a predicate for a Morley sequence). This generalizes to higher dimensions the results of Hrushovski in \cite{Hr} relating failures of $3$-uniqueness (or $4$-amalgamation) to definable groupoids. In future work, we hope to apply these results this to construct a canonical minimal expansion $T^* \supseteq T$ to new sorts which has $n$-amalgamation for \emph{every} integer $n$ by a natural generalization of Hrushovski's method in \cite{Hr} of eliminating ``groupoid imaginaries.''

\subsection*{Background on type amalgamation}

The amalgamation properties considered in this article concern the simultaneous realization of systems of types. By a \emph{type} we mean a collection of formulas $p(\overline{x})$ with free variables from a possibly infinite tuple $\overline{x}$ which is consistent with $T$, and the systems of types which we consider are of the form 
$$\left\{p_s(\overline{x}_s) : s \subsetneq \{1, \ldots, n\} \right\} $$ where $p_s$ is the complete type of a set of $|s|$ elements which are independent over a fixed base set (which we can take to be some realization of $p_\emptyset$). The theory $T$ has $n$-amalgamation (or $n$-existence) if systems of the type above always have consistent unions, and $T$ has $n$-uniqueness if there is never more than one way to amalgamate such a system. In Subsection~1.2 below we give different-looking definitions of the amalgamation properties in terms of elementary maps between \emph{realizations} of the types $p_s$. Applications of the type-amalgamation properties in simple theories include the construction of canonical bases via $3$-amalgamation and a version of the group configuration theorem under the hypothesis of $4$-amalgamation over models (see \cite{DKM}). See also the recent work of Sustretov (\cite{sustretov}) which recasts $4$-amalgamation in terms of Morita equivalence of definable groupoids and gives some applications concerning the interpretability of non-standard Zariski structures.

It was first pointed out by Shelah that any stable $T$ has $2$-uniqueness and $3$-amalgamation, at least if we replace $T$ with the ``harmless'' canonical extension to additional sorts, $T^{eq}$ (see \cite{Sh:c}) . This is now folklore in the field of model theory. More recently, Hrushovski (in \cite{Hr}) developed an intriguing generalization of this, proving that for a stable theory for which $T = T^{eq}$, the following are equivalent: $3$-uniqueness; $4$-amalgamation; any finitary definable connected groupoid is ``eliminable'' (definably equivalent, in a category theory sense, to a group); and any finite internal cover of $T$ is almost split. Using \emph{groupoid imaginaries}, one can then construct a canonical minimal expansion of $T$ to new sorts which has $4$-amalgamation. Hrushovski posed the problem of how to generalize this to higher dimensions.

\subsection*{Summary of results}

Section~1 is devoted to defining all the type amalgamation properties we will use in this paper ($n$-existence, $n$-uniqueness, $B(n)$, and others) and reviewing some basic facts concerning these properties.

In Section~2, we define the general concepts of $n$-ary quasigroupoids and $n$-ary polygroupoids and study some of their basic properties. An $n$-ary polygroupoid is just and $n$-ary quasigroupoid which satisfies an additional ``associativity'' axiom which generalizes the associativity of composition in an ordinary groupoid; see Definition~\ref{associativity} below.

In Section~3, we prove our main result of the paper (Theorem~\ref{polygroup_definability}): in a stable theory with $(\leq n)$-uniqueness but not $(n+1)$-uniqueness, then this failure of $(n+1)$-uniqueness is witnessed by an $n$-ary polygroupoid which is ``almost definable.'' More precisely, to make one of the sorts in the polygroupoid definable, in general we would need to add a unary predicate for a Morley sequence to make one of the sorts definable, but relative to this the rest of the structure is first-order definable in the language of $T$.

Finally, in Section~4 we define a family of canonical examples of connected $n$-ary polygroupoids and study their first-order theories (which we call $T_{G,n}$, where $n \geq 2$ is an integer and $G$ is a finite abelian group).

\subsection*{Previously studied examples}

The polygroupoids described in this paper are closely related to the structure described by Hrushovski in~Example~4.7b of~\cite{tot_cat_struct} and simultaneously extend two other previously-studied families of examples. 

One such family was described by Pastori and Spiga. In \cite{PS}, they construct a family of complete theories $T_n$ (for $n \geq 2$) which are totally categorical, stable, have $k$-existence for every $k \leq n+1$, but do not have $(n+2)$-existence. For the case $n=2$, the theory $T_2$ is, up to interdefinability, exactly Hrushovski's original example of failure of $4$-existence that appeared in~\cite{DKM}.

The theories $T_n$ are (up to bi-interpretability) just the theories $T_{G, n}$ for $G = \mathbb{Z} / {2 \mathbb{Z}}$ as constructed below in Section~4, so our examples can be seen as a natural generalization of those in~\cite{PS} to an arbitrary abelian group. In the notation of~\cite{PS}, there is a basic indiscernible set $\Omega$ (what we would call ``$I$'') and sorts for $[\Omega]^n$ (all $n$-element subsets of $\Omega$) and for $[\Omega]^n \times \mathbb{F}_2$, for the ``fiber elements'' (what we call the sort $P$).

The approach of \cite{PS} is to describe the structures via their permutation groups using the machinery of modules over infinite symmetric groups. We note that it is not immediately clear how to generalize the techniques of \cite{PS} beyond the context of characteristic $2$ since a crucial lemma of that paper, Proposition~2.2, is false over fields of odd characteristic.

On the other hand, Hrushovski in~\cite{Hr} essentially showed that the theories that we call $T_{G,2}$ below (that is, the theory of a connected infinite groupoid with $\textup{Mor}(a,a) \cong G$) canonically witness the failure of $4$-existence in a stable theory. The theories we will define as $T_{G,n}$ in Section~4 below are a simultaneous generalization of these groupoids and the examples $T_n$ of \cite{PS}.

\subsection*{Notation and conventions}

Throughout this paper, $T$ will denote a complete, \textbf{multisorted} theory with elimination of imaginaries ($T = T^{eq}$) and with a monster model $\mathfrak{C} \models T$ (that is, $\mathfrak{C}$ is $\kappa$-saturated and strongly $\kappa$-homogeneous from some ``large'' cardinal $\kappa$ which is bigger than $2^{|T|}$ and any of the other sets under consideration). Unless otherwise specified, all elements and sets come from $\mathfrak{C}$.

Throughout the article, $[n]$ refers to the $n$-element set $\{1, \ldots, n\}$ (and not $\{0, \ldots, n-1\}$, as is more common in algebraic topology).

For any set $A$ and any integer $n \geq 1$, $A^{(n)}$ denotes the set of all ordered tuples from $A^n$ consisting of $n$ distinct elements. $A^{(\leq n)}$ denotes $A^{(1)} \cup \ldots \cup A^{(n)}$.

Often we will refer to finite indexed sets of elements such as $\{a_1, \ldots, a_n\}$. In this context, if $s \subseteq [n]$, then we will write $A_s$ for the set $\{ a_i : i \in s\}$ and we will write $\ov{A}_s$ for $\acl(A_s)$. The capital letters are to distinguish these from the \emph{ordered tuples} $\ov{a}_s$ which we will sometimes refer to (for example, in Definition~\ref{n_system} below).

\section{The type amalgamation properties ($n$-existence and $n$-uniqueness)}

\subsection{Definitions and basic facts}

In this subsection, we briefly review the definitions of the type amalgamation properties that we will use in this paper ($n$-existence or $n$-amalgamation, and $n$-uniqueness). This is to make the current paper more self-contained; more details can be found in \cite{gkk}. Our ``functorial'' approach to the amalgamation properties comes from Hrushovski (\cite{Hr}), and these properties were studied even earlier by researchers in simple theories (for example, in \cite{KKT}).

Everywhere below, we will assume that $T$ is a simple theory so that we have the standard well-behaved notion of independence (nonforking) from model theory.

Let $\mathcal{P}^{-}([n]) = \{s : s \subsetneq [n]\}$. If $S\subseteq \P([n])$ is closed under subsets, we view $S$ as a category, where the objects are the elements of $S$ and morphisms are the inclusion maps. Let $\cC$ be a category whose objects are all algebraically closed subsets of $\C$ and whose morphisms are the elementary maps (which are not always surjective).

\begin{definition}
An \emph{$n$-amalgamation problem} is a functor $\cA: \mathcal{P}^{-}([n]) \rightarrow \cC$. A \emph{solution} to an $n$-amalgamation problem $\cA$ is a functor $\cA': \mathcal{P}([n]) \rightarrow \cC$ that extends $\cA$.

We say that $\cA$ is an amalgamation problem \emph{over} the set $\cA(\emptyset)$ (which is also called the ``base set'').
\end{definition}

\begin{definition}
If $S$ is a subset of $[n]$ closed under subsets and $\cA:S\to \cC$ is a functor, then for $s \subseteq t \in S$, let the \emph{transition map} $$\cA_{s,t} : \cA(s) \rightarrow \cA(t)$$ be the image of the inclusion $s \subseteq t$.  Functoriality implies that $\cA_{t,u} \circ \cA_{s,t} = \cA_{s,u}$ whenever the composition is defined.
\end{definition}

\begin{notation}
If $S$ is a subset of $[n]$ closed under subsets, $\cA:S\to \cC$ is a functor, and $s\subset t\in S$, we use the symbol $\cA_t(s)$ to denote the subset $\cA_{s,t}(\cA(s))$ of $\cA(t)$.

If $\cA:\P([n])\to \cC$ is a functor, then $\cA^-$ denotes the functor $\cA\restriction \Pm([n])$.
\end{notation}

\begin{definition}
Suppose that $S$ is a subset of $\P([n])$ closed under subsets and $\cA:S\to \cC$ is a functor.

\begin{enumerate}

\item We say that $\cA$ is \emph{independent} if for every nonempty $s \in S$, $\left\{\cA_{s} (\{i\}) : i \in s \right\}$ is an $\cA_{s}(\emptyset)$-independent set.

\item We say that $\cA$ is \emph{closed} if for every nonempty $s \in S$, $$\cA(s) = \acl \left( \cup \left\{ \cA_{s} (\{i\}): i \in s \right\} \right).$$

\item We say that $\cA$ is \emph{over $B$} if $\cA(\emptyset) = B$ and every transition map $\cA_{s,t}$ fixes $B$ pointwise.
\end{enumerate}

\end{definition}

\begin{remark}
If the transition maps are all inclusions in a functor $\cA: S \rightarrow \cC$, then $\cA$ is independent if and only if for every $s, t \in \dom(\cA)$, $\cA(t) \nonfork_{\cA(t \cap u)} \cA(u)$.
\end{remark}

\begin{definition}
Two solutions $\cA'$ and $\cA''$ of the $n$-amalgamation problem $\cA$ are \emph{isomorphic} if there is an elementary map $\sigma: \cA'([n]) \rightarrow \cA''([n])$ such that for any $s \subsetneq [n]$, $$\sigma \circ \cA'_{s,[n]} = \cA''_{s,[n]}.$$ They are \emph{isomorphic over $B$} if $\cA'$, $\cA''$ are both over $B$ and the isomorphism $\sigma$ can be chosen so that it fixes $B$ pointwise.

\end{definition}

\begin{definition}
\begin{enumerate}

\item A theory $T$ has \emph{$n$-existence}, equivalently \emph{$n$-amalgamation}, if every closed independent $n$-amalgamation problem has an independent solution.

\item A theory $T$ has \emph{$n$-uniqueness} if whenever a closed independent $n$-amalgamation problem $\cA$ has closed independent solutions $\cA'$ and $\cA''$, then the solutions $\cA'$ and $\cA''$ must be isomorphic. 

\item $T$ has \emph{($< n$)-uniqueness} if it has $k$-uniqueness for every $k$ such that $2 \leq k < n$. Similarly for ($\leq n$)-uniqueness, ($< n$)-amalgamation, \emph{et cetera}.

\end{enumerate}

\end{definition}

Note that the existence of non-forking extensions of types and the independence theorem implies that \emph{any} simple theory has both 2- and 3-existence. In addition, stationarity is equivalent to 2-uniqueness property, so any stable theory has $2$-uniqueness.

Various useful equivalent definitions of $n$-existence and $n$-uniqueness can be found in \cite{gkk} and \cite{Hr}, of which we recall some of them below.

\begin{definition}
Let $T$ be a theory and $A \subseteq \mathfrak{C}$ be a small algebraically closed set.

\begin{enumerate}
\item $T$ has \emph{$n$-existence over $A$} if every closed independent $n$-amalgamation problem $\cA$ over $A$ has an independent solution.
\item $T$ has \emph{$n$-uniqueness over $A$} if every closed independent $n$-amalgamation problem over $A$ has at most one independent solution modulo isomorphism over $A$.
\end{enumerate}

\end{definition}

\begin{remark}
In considering the amalgamation properties, without loss of generality one can restrict attention to closed independent functors $\cA$ in which every transition map fixes $\cA(\emptyset)$ pointwise. Therefore it follows immediately from the definitions that $T$ has $n$-existence if and only if $T$ has $n$-existence over every algebraically closed $A$, and similarly for the uniqueness properties. However, having $n$-existence over $A$ does not necessarily imply having $n$-existence over every algebraically closed $B \supseteq A$, as observed in Remark~3.3 of \cite{ICM_homology}.
\end{remark}

\begin{proposition}
\label{over_model}
(See Section~1 of \cite{DKM}) If $T$ is stable and $M \models T$, then $T$ has $n$-existence over $M$ for every $n$.
\end{proposition}

The next fact below is important for us in the current paper since it immediately implies that to check that $T$ has $n$-existence and $n$-uniqueness for every $n$, it suffices just to check the $n$-uniqueness properties:

\begin{proposition}
\label{uniq_exist}
Let $T$ be stable with ($< n$)-uniqueness (for some $n \geq 3$). Then $T$ has $n$-uniqueness if and only if $T$ has ($n+1$)-amalgamation.
\end{proposition}

\begin{proof}
Mostly, this is straightforward diagram-chasing. See Section~$3$ of \cite{gkk} for a proof, or 4.1 of \cite{Hr} for the left-to-right direction.
\end{proof}

A stable theory may fail $3$-uniqueness, or have ($< n$)-uniqueness but not $n$-uniqueness for any $n \geq 3$; examples of this were given in \cite{PS}, and in Section~4 we will give a more general class of examples via polygroupoids.

Finally, we recall some variants of the amalgamation properties which will be useful to us later. For further discussion and proofs of the facts below, see \cite{gkk}.

\begin{definition}
\label{tuple_bd}
Given a tuple $\overline{a} = (a_1, \ldots, a_k)$ and a set $B$, let $$\bd_B(\overline{a}) := \dcl \left( \bigcup_{i \in \{1, \ldots, k\}} \acl_B(A_{[k] \setminus \{i\}}) \right).$$ If $B = \emptyset$, we omit it.
\end{definition}

\begin{definition}
\label{Bn}
If $A \subseteq \mathfrak{C}$, we say that \emph{$B(n)$ holds over $A$} if for every $A$-independent set $\{a_1, \ldots, a_n\}$, $$\bd_{A \cup \{a_n\}} (a_1, \ldots, a_{n-1}) \cap \acl_A(a_1, \ldots, a_{n-1})= \bd_A(a_1, \ldots, a_{n-1}).$$ The theory $T$ \emph{has $B(n)$} if it has $B(n)$ over every $A \subseteq \mathcal{C}$.
\end{definition}

\begin{fact}
\label{Bn_uniqueness}
(See Proposition~3.14 of \cite{gkk}) If $T$ is stable and has $(n-1)$-uniqueness, then $T$ has $n$-uniqueness if and only if $T$ has $B(n)$. In particular, for stable $T$, $B(3)$ is equivalent to $3$-uniqueness (since any stable theory has $2$-uniqueness).
\end{fact}

We recall an alternative formulation of the property $B(n)$ established in~\cite{gkk}, Lemma 3.3 (with slight changes in the notation for consistency). Here and throughout, ``$\Aut(C / B)$'' denotes the set of all restrictions to $C \subseteq \mathfrak{C}$ of elementary maps $\varphi \in \Aut(\mathfrak{C}/B)$ which map $C$ onto $C$.

\begin{fact}
\label{Bn_equivalent}
In any simple theory $T$ and $n \geq 2$, the following are equivalent:
\begin{enumerate}
\item $T$ has $B(n+1)$ over the set $B$;

\item For any $B$-independent set $\{c_1, \ldots, c_{n+1}\}$, any map $$\varphi \in \Aut(\acl_B(c_1, \ldots, c_n) / \bd_B(c_1, \ldots, c_n))$$ can be extended to some $\tilde{\varphi} \in \Aut(\C / \bd_{B \cup \{c_{n+1}\}}(c_1, \ldots, c_n))$.

\end{enumerate}

\end{fact}

\begin{definition}
\label{rel_nk}
If $A \subseteq \mathfrak{C}$ and $k \geq n$, then $T$ has \emph{relative $(n,k)$-uniqueness over $A$} if for every $A$-independent collection of algebraically closed sets $a_1, \ldots, a_k$ from $\mathfrak{C}^{eq}$, each containing $A$, and for any collection of functions $\{\varphi_u : u \subset_{n-1} \{1, \ldots, k\} \}$ such that:

\begin{enumerate}
\item $\varphi_u$ is an elementary map from $\acl_A(\overline{a}_u)$ onto itself; and
\item for any $v \subsetneq u$, $\varphi_u$ fixes $\acl_A(\overline{a}_v)$ pointwise, 
\end{enumerate}

then the union $$\bigcup_{u \subset_{n-1} \{1, \ldots, k\}} \varphi_u$$ is an elementary map.

\end{definition}

\begin{fact}
\label{skeletal_uniqueness}
(See Lemma~4.4 of \cite{gkk}) If $T$ is stable and has $B(n)$ over $A$ (for $n \geq 3$), then $T$ has relative $(n,k)$-uniqueness over $A$ for every $k \geq n$. In particular, if $T$ is stable and has both $(n-1)$-uniqueness and $n$-uniqueness over $A$, then $T$ has relative $(n,k)$-uniqueness over $A$.
\end{fact}

\begin{lemma}
\label{skeletal_maps}
Suppose that $T$ is stable and has $(\leq n)$-uniqueness. Suppose, moreover, that we are given two independent sets $\{a_1, \ldots, a_k\}$ and $\{b_1, \ldots, b_k\}$ and a system of elementary bijections $$\varphi_s: \ov{A}_s \rightarrow \ov{B}_s$$ indexed by all the $s \subseteq [k]$ with size at most $n-1$ (recalling the notational convention that $\ov{A}_s := \acl(\{a_i : i \in s\})$), and that $t \subseteq s$ implies $\varphi_s$ extends $\varphi_t$. Then there is a single elementary map $$\varphi: \ov{A}_{[k]} \rightarrow \ov{B}_{[k]}$$ which extends every map $\varphi_s$.
\end{lemma}

\begin{proof}
Begin with any elementary map $\varphi_0 : \ov{A}_{[k]} \rightarrow \ov{B}_{[k]}$ such that $\varphi_0(a_i) = b_i$ for every $i \in [k]$. Next, we recursively build a series of elementary bijections $\varphi_1, \ldots, \varphi_{n-1}$ of $\ov{B}_{[k]}$ such that for each $j \leq n-1$, the map $$\varphi_j \circ \varphi_{j-1} \circ \ldots \circ \varphi_1 \circ \varphi_0$$ extends every $\varphi_s$ such that $|s| \leq j$: given a series of such maps $\varphi_1, \ldots, \varphi_{j-1}$ for some $j \leq n-1$, use the fact that $T$ has relative $(j + 1,k)$-uniqueness (by Fact~\ref{skeletal_uniqueness} above) to build the map $\varphi_j$. Finally, let $\varphi = \varphi_{n-1} \circ \varphi_{n-2} \circ \ldots \circ \varphi_0$.

\end{proof}

For the final lemma of this subsection, we note that with a suitably strong amalgamation assumption, we build ``symmetric'' systems of algebraically closed sets which generalize indiscernible sets.


\begin{definition}
\label{n_system}
Suppose $A \cup \{a_1, \ldots, a_k\} \subseteq \mathfrak{C}$ and $1 \leq n \leq k$. Then an \emph{$n$-symmetric system for $\{a_1, \ldots, a_k\}$ over $A$} is a collection of (generally infinite) tuples $\{\ov{a}_s : s \in [k]^{(\leq n)} \}$ such that:

\begin{enumerate}
\item $\ov{a}_s$ is a tuple enumerating $\acl_A(\{a_i : i \in s \})$;
\item For any permutation $\sigma$ of $[k]$, there is an elementary map $\varphi_\sigma \in \Aut(\mathfrak{C} / A)$ such that $$\varphi_\sigma(\ov{a}_s) = \ov{a}_{\sigma(s)}$$ for every $s \in [k]^{(\leq n)}$; and
\item The maps commute when applied to the tuples $\ov{a}_s$: for any $s \in [k]^{(\leq n)}$ and any two permutations $\sigma$ and $\tau$ of $[k]$, $$\varphi_\sigma (\varphi_\tau(\ov{a}_s)) = \varphi_{\sigma \circ \tau}(\ov{a}_s).$$

\end{enumerate}



\end{definition}

\begin{lemma}
\label{n_system_existence}
Suppose that $T$ has $\le n$-uniqueness and let $\{a_1,\dots,a_{k}\}$ be a Morley sequence over $A=\acl(A)$, where $k \geq n$. Then there is an $(n-1)$-symmetric system for $\{a_1, \ldots, a_k\}$ over $A$.

\end{lemma}

\begin{proof}
First, we set some terminology for dealing with finite ordered tuples $s \in [k]^{(\leq n-1)}$. We abuse notation and write ``$i \in s$'' to mean that $i$ occurs in the tuple $s$. Given two tuples $s = (a_1, \ldots, a_m)$ and $t = (b_1, \ldots, b_m) $ of the same length $m$, let $\sigma_{s,t} : \{a_1, \ldots, a_m\} \rightarrow \{b_1, \ldots, b_m\}$ be the bijection defined by $\sigma_{s,t}(a_i) = b_i$. We say that $s'$ is a \emph{subtuple of $s$} (and write ``$s' \subseteq s$'') if every $i \in s'$ also occurs in $s$, and furthermore any pair of elements $i,j$ which occur in $s'$ also occur in $s$ in the same order (so, for example, $(1,3)$ is a subtuple of $(1,2,3)$ but $(3,1)$ is not).

Now we establish:

\begin{claim}
\label{commuting_sys}
There is a system of tuples $\{ \ov{a}_s : s \in [k]^{(\leq n-1)} \}$ and a system of elementary bijections $$\varphi_{s,t} : \ov{a}_s \rightarrow \ov{a}_t$$ for every pair $(s,t)$ of elements of $[k]^{(\leq n-1)}$ such that $|s| = |t|$, such that:

\begin{enumerate}
\setcounter{enumi}{5}
\item $\ov{a}_s$ enumerates all of the elements of $\acl_A(\{a_i : i \in s \})$;
\item $\varphi_{s,t}$ fixes the base $A$ pointwise;
\item if $|s| =|t|$, $s'$ is a subtuple of $s$, and $t' = \sigma_{s,t}(s')$, then the function $\varphi_{s', t'}$ is equal to the restriction of $\varphi_{s,t}$ to $\ov{a}_{s'}$;
\item $\varphi_{s,s}$ is the identity map on $\ov{a}_s$; and
\item if $|s| = |t| = |u|$, then $\varphi_{s,u} = \varphi_{t,u} \circ \varphi_{s,t}$.
\end{enumerate}

\end{claim}

\begin{proof}
We build the tuples $\ov{a}_s$ and the maps $\varphi_{s,t}$ by recursion on $|s|$. For $|s| = 1$, this is easy, using the fact that $\{a_1, \ldots, a_k\}$ is Morley over $A$. For the recursion step, suppose that we have defined both the tuples $\ov{a}_s$ and the maps $\varphi_{s,t}$ for every $s \in [k]^{(\leq m-1)}$, where $1 < m \leq n-1$. For the purposes of this proof only, we let $[m]$ denote not the set $\{1, \ldots, m\}$ but rather the ordered tuple $(1, \ldots, m)$. First let $\ov{a}_{[m]}$ be any enumeration of $\acl_A(\{a_i : 1 \leq i \leq m \})$. By Lemma~\ref{skeletal_maps}, for any $s \in [k]^{(m)}$, there is an elementary map $$\varphi_{[m], s} : \acl_A(\{a_i : 1 \leq i \leq m \}) \rightarrow \acl_A(\{a_i : i \in s \})$$ such that $\varphi_{[m], s}$ extends every map $\varphi_{s', \sigma_{[m], s}(s')}$ where $s' \subsetneq [m]$. In the special case where $s = [m]$ we can assume that $\varphi_{[m],[m]}$ is the identity map (since we are assuming inductively that condition~(9) holds for subtuples of $[k]$ of length less than $m$). For any $s \in [k]^{(m)}$, define $\ov{a}_s$ to be the image of $\ov{a}_{[m]}$ under the map $\varphi_{[m], s}$, and if $s, t$ are any two elements of $[k]^{(m)}$, let $\varphi_{s,t} := \varphi_{[m], t} \circ \left(\varphi_{[m], s}\right)^{-1}$. The fact that these maps $\varphi_{s,t}$ form a commuting system (conditions~(9) and (10)) is immediate. To check condition (8), suppose that $s' \subsetneq s \in [k]^{(m)}$, $t \in [k]^{(m)}$, and $t' := \sigma_{s,t}(s')$. Then by the definitions, $$\varphi_{s,t}(\ov{a}_{s'}) = \varphi_{[m], t} \left[ \left( \varphi_{s, [m]} \right) (\ov{a}_{s'}) \right] = \varphi_{[m], t} \left[ \left( \varphi_{s', \sigma_{s, [m]}(s')} \right) (\ov{a}_{s'}) \right] $$ $$ =  \varphi_{[m], t}\left[ \ov{a}_{\sigma_{s, [m]}(s')} \right]$$

$$ = \varphi_{\sigma_{s, [m]}(s'), t'} \left[ \ov{a}_{\sigma_{s, [m]}(s')} \right]$$ $$= \varphi_{s', t'}(\ov{a}_{s'}) = \ov{a}_{t'}.$$
\end{proof}

Claim~\ref{commuting_sys} gives the tuples $\ov{a}_s$ for $s \in [k]^{(\leq n-1)}$. For any permutation $\sigma$ is of $\{1, \ldots, k\}$, use Lemma~\ref{skeletal_maps} gives an elementary map $\varphi_\sigma$ extending every map $\varphi_{s,t}$ as in the Claim such that $s \in [k]^{(\leq n-1)}$ and $t = \sigma(s)$. At this point, it is clear that this forms an $(n-1)$-symmetric system.

\end{proof}

\section{$n$-ary polygroupoids}

This section is devoted to the study of \emph{$n$-ary quasigroupoids} (Definition~\ref{quasigroupoid}) and \emph{$n$-ary polygroupoids} (Definition~\ref{associativity}). We note that these structures are a simultaneous generalization of groupoids (which corresponds to the case $n=2$) and $n$-ary groups (as studied by Emil Post in \cite{post}). Groupoids generalize groups in the sense that the composition operation is only partially defined, since we need a compatibility condition $s(g) = t(f)$ to hold in order to define $g \circ f$. On the other hand, $n$-ary groups (or polyadic groups) have an operation which is totally defined, but $n$-ary instead of binary. Our $n$-ary groupoids have an operation (given by the $Q$ relation below) which is $n$-ary and partially defined.

There may be connections with Jacob Lurie's higher topos theory which it would be good to clarify further in future research. In particular, \emph{connected} $n$-ary quasigroupoids (see Definition~\ref{connected}) look quite similar to $(\infty, 1)$-groupoids in the sense of Definition~1.1.2.4 of \cite{lurie}, or more precielly, Kan complexes in which $n$-dimensional horns have unique fillers.

There is no model theory in this section; the connection with failures of amalgamation properties will be established in Section~3.

Let us start with a preliminary definition.

\begin{definition}
\label{comp_tuples}
Let $M=(P_1,\dots,P_n,\pi^2,\dots,\pi^n)$ be a structure with sorts $P_i$, $i=1,\dots,n$ and
functions $\pi^k:P_k\to (P_{k-1})^k$, $k=2,\dots,n$. We use the symbol $\pi^k_i(f)$ to refer to the $i$th element of the tuple $\pi^k(f)$.

We say that a tuple $(f_1,\dots,f_{k+1})\in (P_k)^{k+1}$ is \emph{compatible} if
\begin{enumerate}
\item
$k=1$ and $f_1\ne f_2$; or
\item
$k\ge 2$ and $\pi^k_i(f_j)=\pi^k_{j-1}(f_i)$ for all $1\le i<j\le k+1$.
\end{enumerate}
\end{definition}

If the images of the projection maps $\pi^i$ are always compatible tuples, then we can iteratively apply the projections to define a concept of an element $f \in P_i$ being ``over'' an $i$-tuple from $P_1$:

\begin{definition}
\label{over_P1}
Suppose that $M = (P_1, \ldots, P_n, ,\pi^2,\dots,\pi^n)$ is a structure with sorts $\{P_1, \ldots, P_n\}$ and
functions $\pi^k:P_k\to (P_{k-1})^k$ for $k=2,\dots,n$, and suppose that for every $f \in P_k$, the image $\pi^k(f)$ forms a compatible tuple in the sense of the previous definition. Then for any $i \in \{2, \ldots, n\}$ and any $f \in P_i$, we say that \emph{$f$ is over $(a_1, \ldots, a_i) \in (P_1)^{(i)}$} if:

\begin{enumerate}
\item $i = 2$ and $(a_1, a_2) = \pi^2(f)$; or
\item $i > 2$ and for every $j \in \{1, \ldots, i\}$, $\pi^i_j(f)$ is over $(a_1, \ldots, \widehat{a}_j, \ldots, a_i)$.
\end{enumerate}

For any $(a_1, \ldots, a_i) \in I^{(i)}$, we denote by $P_i(a_1, \ldots, a_i)$ the set of all $f \in P_i$ which are over $(a_1, \ldots, a_i)$. If $f \in P_i(a_1, \ldots, a_i)$, we also write ``$\pi(f) = (a_1, \ldots, a_i)$.''

\end{definition}

The point of assuming that $\pi^i(f)$ is always a compatible tuple in the definition above is that this easily implies that for every $i \in \{2, \ldots, n\}$, the sort $P_i$ is the disjoint union of all the ``fibers'' $P_i(a_1, \ldots, a_i)$ where $(a_1, \ldots, a_i) \in (P_1)^{(i)}$.

Below, we will be dealing with $n$-sorted structures, with sorts $P_1,\dots,P_n$. We found it convenient to use separate symbols $I$ and $P$ for the most frequently used sorts $P_1$ and $P_n$, respectively.

\begin{definition}
\label{quasigroupoid}
If $n \geq 2$, an \emph{$n$-ary quasigroupoid}\footnote{The ``quasi-'' prefix accords with the terminology from abstract algebra that a ``quasigroup'' is a set with a binary operation which is divisible but not necessarily associative.} is a structure $\mathcal{H} = (I, P_2, \ldots, P_{n-1}, P, Q)$ with $n$ disjoint sorts $I = P_1, P_2, \ldots, P_n = P$ equipped with an $(n+1)$-ary relation $Q \subseteq P^{n+1}$ and a system of maps $\langle \pi^k : 2 \leq k \leq n \rangle$ satisfying the following axioms:

\begin{enumerate}
\item
For each $k \in \{2, \ldots, n\}$, the function $\pi^k$ maps an element $f\in P^k$ to a $k$-tuple
$(\pi^k_1(f),\dots, \pi^k_k(f))$.

\item (Coherence) For any $k \in \{2, \ldots, n\}$, for each $f\in P_k$, the tuple $\pi^k(f)$ is compatible.

\item
(Compatibility and $Q$)
If $Q(p_1, \ldots, p_{n+1})$ holds, then $(p_1, \ldots, p_{n+1})$ is a compatible $(n+1)$-tuple
of elements of $P$.

\item (Uniqueness of horn-filling) Whenever $Q(p_1, \ldots, p_{n+1})$ holds, then for any $i \in \{1, \ldots, {n+1}\}$, $p_i$ is the unique element $x \in P$ which satisfies $$Q(p_1, \ldots, p_{i-1}, x, p_{i+1}, \ldots, p_{n+1}).$$
\end{enumerate}

\end{definition}

For any set $I$ and any integer $n \geq 1$, we denote by $I^{(n)}$ the set of all ordered $n$-tuples of pairwise distinct elements from $I$. The coherence axiom guarantees, in
particular, that every element in $P_j$ is supported, in the sense made precise 
below, by $j$ distinct elements from $I$. 




\begin{definition}
\label{support}
Let $\mathcal{H} = (I, \ldots, P, Q)$ be an $n$-ary quasigroupoid and suppose that $X \subseteq (I \cup P_2 \cup \ldots \cup P)$. 

\begin{enumerate}


\item The \emph{support of $X$}, denoted $\supp(X)$, is the subset of $I$ generated from $X$ by applying the projection maps $\langle \pi^i_j : 2 \leq i \leq n, 1 \leq j \leq i \rangle$.

\item 
The \emph{closure} $\cl(X)$ is $$\cl(X) := \supp(X) \cup \bigcup_{2 \leq i \leq n} \bigcup_{w \in \supp(X)^{(i)}} P_i(w) $$

\end{enumerate}
\end{definition}

\begin{definition}
\label{local_finiteness}
The $n$-ary quasigroupoid $\mathcal{H} = (I, \ldots, P, Q)$ is \emph{locally finite} if for every $w \in I^{(n)}$, the set $P(w)$ is finite.
\end{definition}

Note that for any $n$-ary quasigroupoid $\mathcal{H} = (I, \ldots, P, Q)$ and any $X \subseteq (I \cup P_2 \cup \ldots \cup P)$, the set $\cl(X)$ naturally carries the structure of a ``sub-quasigroupoid'' of $\mathcal{H}$, where the $I$-sort is $\supp(X)$, the remaining sorts are the fibers in $\mathcal{H}$ above tuples from $\supp(X)$, and the $Q$-relation is the same as in $\mathcal{H}$ only restricted to tuples from $P \cap \cl(X)$.

\begin{definition}
\label{connected}
Let $\mathcal{H} = (I, \ldots, P, Q)$ be an $n$-ary quasigroupoid. We call $\mathcal{H}$ \emph{connected} if it satisfies both of the following conditions:
\begin{enumerate}
\item If $i \in \{1, \ldots, n-1\}$, then any compatible tuple $(p_1, \ldots, p_{i+1})$ from $P_i$ is in the image of $\pi^{i+1}$.\footnote{Note that when $i = 1$, then this simply says that every ordered pair $(a_1, a_2)$ of distinct elements from $I$ is in the image of $\pi^2$.}
\item (Existence of horn fillers) Suppose that $i \in \{1, \ldots, n+1\}$ and $\langle p_j : 1 \leq j \leq n+1 \rangle$ is a tuple of elements from $P$ which is compatible with respect to $\pi^n$. Then there is an element $p'_i \in P$ such that $Q$ holds of $(p_1, \ldots, p'_i, \ldots, p_{n+1})$.
\end{enumerate}

\end{definition}

The following two claims show that every ``partially compatible'' sequence of elements of $P$ can be completed to a compatible sequence of length $n+1$ that satisfies $Q$ and that we can select the compatible elements satisfying the relation $Q$ on all but possibly one $n$-dimensional face of an $(n+1)$-dimensional tetrahedron. 

\begin{definition}
\label{compatible}
Let $\mathcal{H} = (P_1, \ldots, P_n, Q)$ be an $n$-ary quasigroupoid. Let $A_i\subset P_i$, $i=1,\dots,n$ be given sets. We say that the collection $\{A_i\mid i=1,\dots,n\}$ is \emph{closed under projections} if for all $k=1,\dots,n-1$, for all $a\in A_{k+1}$ and all $1\le j\le k$ we have $\pi^{k+1}_j(a)\in A_k$.

We say that a collection $\{A_i\subset P_i \mid i=1,\dots,n\}$ is a \emph{compatible system} if it is closed under projections and each $i=2,\dots, n$ and every tuple $w\in A_1^{(i)}$, the set $P_i(w)\cap A_i$ has at most one element.
\end{definition}

\begin{claim}\label{face_fill}
Let $\mathcal{H} = (P_1, \ldots, P_n, Q)$ be a connected $n$-ary quasigroupoid. Let 
$\{A_i\mid i=1,\dots,n\}$ be a compatible system such that $|A_1|\ge n+1$. 

(1) Let $2\le i\le n$ and let $f,g\in A_i$ be elements that project onto tuples 
$u,v\in A_1^{(i)}$ such that $\bd_ju=\bd_kv$ for some $1\le j,k\le i$, where ``$\bd_j$'' is the operation which deletes the $j$th entry of an ordered tuple. Then $\pi^i_j(f)=\pi^i_k(g)$.

(2) There is a maximal compatible system $\{B_i\mid i=1,\dots,n\}$ extending 
$\{A_i\mid i=1,\dots,n\}$ that projects onto $A_1$. That is, there are sets $B_i$ satisfying 
\begin{itemize}
\item[(a)]
$B_1=A_1$ and $B_i\supset A_i$ for $2\le i\le n$;
\item[(b)] 
for each $i=2,\dots, n$ and for every tuple $w\in A_1^{(i)}$, the set 
$P_i(w)\cap A_i$ has exactly one element.
\end{itemize}
\end{claim}

\begin{proof}
(1) Let $w=\bd_j u = \bd _k v$ be the common subtuple. Since $\{A_i\mid i=1,\dots,n\}$ is closed under projections, we have $\pi^i_j(f),\pi^i_k(g)\in A_{i-1}$ are the elements that project onto the same tuple $w$. Therefore, $\pi^i_j(f)=\pi^i_k(g)$.

(2) By induction on $j\ge 1$, we construct the sets $B_j\supset A_j$ such that 
$\{B_1,\dots,B_j,A_{j+1},\dots,A_n\}$ form a compatible system and for all 
$2\le j \le n$ and all $w \in A_1^{(j+1)}$, there is a compatible 
tuple $(f_1,\dots,f_{j+1})$ of elements of $B_j$ such that $f_k$ projects onto 
$\bd _k(w)$ for all $k=1,\dots,j+1$.

Let $B_1:=A_1$. Suppose the sets $B_1,\dots,B_j$ have been constructed.
To construct $B_{j+1}$, enumerate all the tuples in the set $A_1^{(j+1)}=\{w_i\mid i<\alpha\}$. 
For each tuple $w_i$, if there is an element $f\in A_{j+1}$ that projects onto $w_i$, let $g_i:=f$; otherwise, take the compatible sequence $(f_1,\dots,f_{j+1})$ in $B_j$ such that $f_k$ projects onto 
$\bd _k(w_i)$. By connectedness of $\mathcal{H}$, there is an element $g_i\in P_{j+1}$ such that $\pi^{j+1}(g_i)=(f_1,\dots,f_{j+1})$. Let $B_{j+1}:=\{g_i\mid i<\alpha\}$. By construction, $B_{j+1}\supseteq A_{j+1}$; for every $w\in A_1^{(j+1)}$, there is exactly one element in the set $B_{j+1}\cap P(w)$; and for any sequence $(a_1,\dots,a_{j+2})$ of elements of $A_1$, the elements of $B_{j+1}$ that project onto subsequences of length $j+1$ are automatically compatible by (1).
\end{proof}

\begin{claim}\label{horn_fill}
Let $\mathcal{H} = (I, \ldots, P, Q)$ be a connected $n$-ary quasigroupoid. Given distinct $a_1,\dots,a_{n+2}$ in $I$, there exist elements $\{p^i_j\mid 1\le i\le n+2,\ 1\le j\le n+1\}\subset P$ such that
\begin{enumerate}
\item
$Q(p^i_1,\dots,p^i_{n+1})$ for all $i=2,\dots, n+2$;
\item
$\supp(\{p^i_j\mid 1\le j\le n+1\})=\{a_1,\dots,\widehat{a_i},\dots,a_{n+2}\}$
 for all $i=1,\dots,n+2$;
\item
$p^i_j=p^{j+1}_i$ for all $1\le i\le j\le n+1$.
\end{enumerate}
Moreover, if $\{p_1,\dots,p_{n+1}\}\subset P$ have the support $\{a_1,\dots,a_{n+1}\}$,
$p_j$ projects onto $\bd_j(a_1,\dots,a_{n+1})$, and $Q(p_1,\dots,p_{n+1})$ holds,
then we may choose $p^{n+2}_j=p_j$ for $j=1,\dots,n+1$.
\end{claim}

\begin{proof}
It suffices to prove the ``moreover'' clause of the above lemma. Let $\{A_1,\dots,A_n\}$ be the 
compatible system generated by the elements $\{p_1,\dots,p_{n+1}\}$ (that is a system closed under projections such that $A_n=\{p_1,\dots,p_{n+1}\}$. Apply Claim~\ref{face_fill} to the compatible system $\{A_1\cup\{a_{n+2}\},A_2,\dots,A_n\}$; let $\{B_1,\dots,B_n\}$ be the maximal compatible extension. For
$2\le i\le n+2$ and $2\le j\le n+1$, let $p^i_j$ be the unique element of $B_n$ that projects onto
$\bd_j(\bd_i(a_1,\dots,a_{n+2}))$. For $i\ge 2$, by Definition~\ref{connected}, we may choose $p^i_1\in P_n(\bd_i(a_2,\dots,a_{n+2}))$ such that $Q(p^i_1,\dots,p^i_{n+1})$ holds. Finally, let
$p^1_j:=p^{j+1}_1$ for $j=1,\dots,n+1$. This gives the required collection of elements.
\end{proof}

\begin{definition}
\label{associativity}
If $\mathcal{H} = (I, \ldots, P, Q)$ is an $n$-ary quasigroupoid, we say that $\mathcal{H}$ is an \emph{$n$-ary polygroupoid} if it satisfies the following condition:

(Associativity) Suppose that 
$\{p^i_j \mid 1 \leq i \leq n+2,\ 1\le j\le n+1\}$ is a 
collection of elements in $P$ such that for each $i=1,\dots,n+2$
the elements $\{p^i_j\mid 1\le j\le n+1\}$ are compatible
and such that $p^i_j=p^{j+1}_i$ for all $1\le i\le j\le n+1$.

For each $\ell=1,\dots,n+2$, if $Q(p^i_1,\dots,p^i_{n+1})$ 
holds for all $i\in \{1,\dots,n+2\}\setminus \{\ell\}$, then 
$Q(p^\ell_1,\dots,p^\ell_{n+1})$ holds.

\end{definition}

It will be useful for later to note that the conclusion of Claim~\ref{horn_fill} can be strengthened, if we have associativity:

\begin{remark}
\label{Q_solutions}
If $\mathcal{H}$ is an $n$-ary polygroupoid, then for the set $\{p_j^i\mid 1\le i\le n+2,\ 1\le j\le n+1\}$ constructed in Claim~\ref{horn_fill}, we also have $Q(p^1_1,\dots,p^1_{n+1})$.
\end{remark}

\begin{definition}
\label{group_action}
Let $G$ be a group and $\mathcal{H} = (I, \ldots, P, Q)$ an $n$-ary quasigroupoid. A \emph{regular action of $G$ on $\mathcal{H}$} is a group action on the set $P$ such that:

\begin{enumerate}
\item For any $p \in P$ and any $g \in G$, $\pi^n(g.p) = \pi^n(p)$;
\item $G$ acts regularly (transitively and faithfully) on each fiber $P_n(w)$, where $w$ is a compatible $(n+1)$-tuple from $P$; and
\item If $Q(p_1, \ldots, p_{n+1})$ holds, then for any $g \in G$ and any $i \in \{1, \ldots, n\}$, we also have $$Q(p_1, \ldots, g.p_i, g.p_{i+1}, \ldots, p_{n+1})$$ (where only the two elements $p_i$ and $p_{i+1}$ have been acted upon by $g$).
\end{enumerate}

\end{definition}

It turns out that the group $G$ in the definition above is always abelian.

\begin{proposition}
\label{alt_2}
Let $\mathcal{H}$ be an $n$-ary quasigroupoid (for $n \geq 2$) equipped with a regular action of the group $G$, and assume that $Q$ holds on at least one $(n+1)$-tuple.  Then:
\begin{enumerate}
\item
For any $p_1,\dots, p_{n+1}\in P$ and any $i < j$, 
$$Q(p_1, \ldots, p_{n+1}) \Leftrightarrow Q(p_1, \ldots,  g. p_i,  \ldots,  g^{(-1)^{j-i+1}}.f_j,  \ldots, p_{n+1}).$$
\item
The group $G$ is abelian (so we write the group additively below).
\item
If $Q(p_1, \ldots, p_{n+1})$ holds, then for any $(g_1, \ldots, g_{n+1}) \in G^{n+1}$, $$Q(g_1.p_1, g_2.p_2, \ldots, g_{n+1} .p_{n+1}) \Leftrightarrow \sum_{i=1}^{n+1} (-1)^i g_i = 0.$$
\end{enumerate}
\end{proposition}

\begin{proof}
The first statement is straightforward, by property~(3) of Definition~\ref{group_action} and induction on $j-i$.

To establish claim (2), note that $Q(p_1,\dots,p_{n+1})$ implies, for all $g,h\in G$
$$
Q(g.p_1,g.p_2,p_3,\dots,p_{n+1}),\quad Q(h.p_1,p_2,h^{-1}.p_3,\dots,p_{n+1})
$$ 
and, acting by $h$ and $g$ respectively on the sets of elements above we get 
$$
Q((hg).p_1,g.p_2,h^{-1}.p_3,\dots,p_{n+1}),\quad Q((gh).p_1,g.p_2,h^{-1}.p_3,\dots,p_{n+1}).
$$
It follows from Axiom~(2) in the definition of a polygroupoid that $(hg).p_1 = (gh).p_1$ and therefore $hg=gh$ by regularity of the action.

The third statement is immediate from (1).
\end{proof}

From now on, whenever $G$ is a group acting regularly on a quasigroupoid $\mathcal{H}$, then we will write both the operation of $G$ and its action on $\mathcal{H}$ additively.

\subsection{Actions and associativity}

The next goal (Theorem~\ref{associativity_action} below) is to show that any $n$-ary quasigroupoid $\mathcal{H}$ with a regular action by a group $G$ which satisfies a certain homogeneity assumption (which will always hold in our model-theoretic application) is necessarily almost associative: the associativity law for $Q$ holds, possibly modulo ``twisting'' one coordinate by a fixed element of $G$. Since the twisted version $Q_g$ of $Q$ is interdefinable with $Q$ over $G$, this can be used later to obtain definable polygroupoids from definable $n$-ary quasigroupoids.



\begin{definition}
\label{homog_action}
Let $\mathcal{H} = (I, \ldots, P, Q)$ be a connected $n$-ary quasigroupoid. Suppose that $G$ acts regularly on $\mathcal{H}$. We say that \emph{the action is $(n+2)$-homogeneous} if whenever $(a_1, \ldots, a_{n+2})$ and $(b_1, \ldots, b_{n+2})$ are tuples from $I^{(n+2)}$ and $(A_1, \ldots, A_{n-1})$, $(B_1, \ldots, B_{n-1})$ are maximal compatible systems going up to $P_{n-1}$ such that $A_1 = \{a_1, \ldots, a_{n+2}\}$ (respectively, $B_1 = \{b_1, \ldots, b_{n+2} \}$), then there is an isomorphism $f: \cl(\overline{a}) \rightarrow \cl(\overline{b})$ such that

\begin{enumerate}
\item $f(a_i) = b_i$ for every $i \in \{1, \ldots, n+2\}$;
\item $f(A_k) = B_k$ for every $k \in \{1, \ldots, n-1\}$; and
\item $f$ commutes with the action of $G$.
\end{enumerate}


\end{definition}


\begin{theorem}
\label{associativity_action}
Let $\mathcal{H} = (I, \ldots, P, Q)$ be a connected $n$-ary quasigroupoid such that $|I| \geq n+3$, and $\mathcal{H}$ is equipped with an $(n+2)$-homogeneous regular action of a group $G$.

\begin{enumerate}
\item 
If $n$ is odd, then there is some $g \in G$ such that if we define a new relation $Q_g$ on $P$ by $$Q_g(p_1, \ldots,  p_{n+1}) \Leftrightarrow Q(p_1, \ldots, p_n, p_{n+1} - g),$$ then the quasigroupoid $\mathcal{H}_g = (I, \ldots, P, Q_g)$ is a polygroupoid.
\item 
If $n$ is even, then $\mathcal{H}$ is a polygroupoid.
\end{enumerate}
\end{theorem}

\begin{proof}
Let $\rho:=(p_1,\dots,p_{n+1})$ be a compatible tuple of elements of $P$. Define the \emph{defect} of $\rho$, denoted $d(\rho)$, to be the unique element $g\in G$ such that $Q(p_1,\dots,p_n,p_{n+1}+g)$ holds. Let $\tau=(\rho^1,\dots,\rho^{n+2})$ be a sequence of compatible tuples $\rho^i=(p^i_1,\dots,p^i_{n+1})$ such that $p^i_j=p^{j+1}_i$ for $1\le i\le j\le n+1$. Define the defect of $\tau$ to be $d(\tau):=\sum_{i=1}^{n+2} (-1)^i d(\rho^i)$. 

It is clear that for every $g\in G$, there is a compatible tuple $\rho$ with $d(\rho)=g$. A key observation is that the defect of a \emph{sequence} of compatible tuples has to be unique.

\begin{claim}
\label{Q_defect}
Let $\mathcal{H} = (I, \ldots, P, Q)$ be a connected $n$-ary quasigroupoid, $|I| \geq n+2$, and suppose that $\mathcal{H}$ is equipped with an $(n+2)$-homogeneous regular action of a group $G$. There is a single element $g\in G$ such that $d(\tau)=g$ for every sequence $\tau=(\rho^1,\dots,\rho^{n+2})$ of compatible tuples $\rho^i=(p^i_j)_{1\le j\le n+1}$ such that $p^i_j=p^{j+1}_i$ for $1\le i\le j\le n+1$.
\end{claim}

\begin{proof}
By Claim~\ref{horn_fill}, there is at least one sequence $\tau=(\rho^1,\dots,\rho^{n+2})$ of compatible tuples, $\rho^i=(p^i_j)_{1\le j\le n+1}$, such that $p^i_j=p^{j+1}_i$. Let $g:=d(\tau)$. We will show that $d(\tilde\tau)=g$ for any other sequence $\tilde\tau=(\tilde\rho^1,\dots,\tilde\rho^{n+2})$ of compatible tuples, $\tilde\rho^i=(\tilde p^i_j)_{1\le j\le n+1}$,  such that $\tilde p^i_j=\tilde p^{j+1}_i$ for $1\le i\le j\le n+1$. Now given such a second sequence $\widetilde{\tau}$ of tuples, we have the two compatible systems $(A_1, \ldots, A_{n-1})$ where $A_k$ is the image of $\tau$ in $P_k$ under the $\pi$ maps, and $(B_1, \ldots, B_{n-1})$, which is similarly defined using images of $\widetilde{\tau}$ under the $\pi$ maps. By $(n+2)$-homogeneity, there is an isomorphism $f: \cl(\tau) \rightarrow \cl(\widetilde{\tau})$ which preserves the action of $G$, and clearly $d(f(\tau)) = d(\tau)$. Thus to prove the general case of Claim~\ref{Q_defect}, it is sufficient to consider the case where $\pi^n(p^i_j) = \pi^n(\tilde p^i_j)$ for every $i, j$.

By the transitivity of the action of $G$ on $\pi^n$-fibers, there is a system of elements $g^i_j\in G$ such that $\tilde p^i_j=p^i_j+g^i_j$. Note that we have $g^i_j=g^{j+1}_i$. Then for each $i=1,\dots,n+2$ we have, by the definition of $d(\rho^i)$:
$$
Q(p^i_1,\dots,p^i_n,p^i_{n+1}+d(\rho^i))\quad \textrm{and}\quad 
Q(p^i_1+g^i_1,\dots,p^i_n+g^i_n,p^i_{n+1}+g^i_{n+1}+d(\tilde \rho^i)).
$$
Thus, by Proposition~\ref{alt_2}(3), we get $(-1)^{n}(d(\tilde\rho^i)-d(\rho^i))=\sum_{j=1}^{n+1} (-1)^jg^i_j$ for each $i=1,\dots,n+2$. The equality $d(\tilde\tau)=d(\tau)$ now follows from the standard ``$\bd^2=0$'' combinatorics:
$$
(-1)^n(d(\tilde\tau)-d(\tau))=\sum_{i=1}^{n+2}(-1)^i\sum_{j=1}^{n+1}(-1)^j g^i_j
=\sum_{i=1}^{n+1}\sum_{j=i}^{n+1}(-1)^{i+j}g^i_j+ \sum_{i=2}^{n+2}\sum_{j=1}^{i-1}(-1)^{i+j}g^i_j=0;
$$
the last equality follows since
$$
\sum_{i=1}^{n+1}\sum_{j=i}^{n+1}(-1)^{i+j}g^i_j=-\sum_{i=1}^{n+1}\sum_{j=i}^{n+1}(-1)^{i+j+1}g^{j+1}_i
=-\sum_{i=2}^{n+2}\sum_{j=1}^{i-1}(-1)^{i+j}g^i_j.
$$
\end{proof}

Let us call the element $g$ constructed in the above claim \emph{the defect of $Q$}. Note that $Q$ is associative if and only if the defect of $Q$ is~0. The rest of the argument proceeds as follows. In the case when $n$ is odd and $g$ is the defect of $Q$, we show that the relation $Q_g$ has defect~0. In the case when $n$ is even, we show that the defect of $Q$ must be~0.

(1) Let $n$ be odd and let $g$ be the defect of $Q$. Let $\{p^i_j\mid 1\le i\le n+2,\ 1\le j\le n+1\}$ be elements constructed in Claim~\ref{horn_fill}. Then we have 
$$
Q(p^1_1,\dots,p^1_n,p^1_{n+1}-g) 
\quad\textrm{and}\quad 
Q(p^i_1,\dots,p^i_{n+1}), \textrm{ for }i=2,\dots, n+2.
$$
Define $Q_g$ by 
$$
Q_g(p_1,\dots,p_{n+1})\iff Q(p_1,\dots,p_{n+1}-g).
$$
Note that the quasigroupoid $\mathcal{H}_g = (I, \ldots, P, Q_g)$ satisfies the assumptions of Claim~\ref{Q_defect}. Therefore, it is enough to check that the defect $d_{g}$ of $Q_g$ is 0 for the sequence $\tau=(\rho^1,\dots,\rho^{n+2})$, where $\rho^i=(p^i_1,\dots,p^i_{n+1})$. We see that $d_g(\rho^1)=0$ and $d_g(\rho^i)=g$ for $i=2,\dots,n+2$. Thus, the defect $d_g(\tau)=0+\sum_{i=2}^{n+2}(-1)^i g=0$; the last equality holds since the sum has an even number of terms.

(2)  
Let $a_1, \ldots, a_{n+3}$ be any $(n+3)$ distinct elements of $I$. 
First, using Claim~\ref{face_fill}, we find a maximal compatible system $(B_1,\dots,B_n)$ such that $B_1=\{a_1,\dots,a_{n+3}\}$. This gives the tuples $\tau^k=(\rho^{k,1},\dots,\rho^{k,n+2})$, where $\rho^{k,j}$ is the sequence of $n+1$ compatible tuples in $B_n$ that project onto the subsequences of $\bd_j(\bd_k(a_1,\dots,a_{n+3}))$
of length $n$. By this definition, the $i$th element of $\tau^k$ is equal to the $k$th element of $\tau^{i+1}$ for $1\le k\le i\le n+2$. 

Now we consider the element $h=\sum_{k=1}^{n+3}(-1)^k d(\tau^k)$. We calculate $h\in G$ in two ways. On the one hand, the $i$th term in $d(\tau_k)$ is equal to $k$th term in $d(\tau^{i+1})$ up to the sign $(-1)^{i+k}$. Thus, in $h$ these terms appear with the opposite signs, and we have $h=0$. On the other hand, by the Claim above, all the defects $d(\tau^k)$ are equal to the same value $g$. Therefore we can cancel out all of the constituent terms of $h$ in pairs except for one (since $n+3$ is odd), and $h = \pm g$. Thus $g=0$, which implies that $Q$ is associative.
\end{proof}

\section{Failure of $(n+1)$-uniqueness implies interpretability of an $n$-ary polygroupoid over a Morley sequence}

In this section, we will establish the main theorem of the paper (Theorem~\ref{polygroup_definability}): a nontrivial connected finitary $n$-ary polygroupoid can be defined (relative to a Morley sequence) in any stable theory which has $k$-uniqueness for all $k \in \{2, \ldots, n\}$  but which fails $(n+1)$-uniqueness. This will be done in two steps: first, we show how to define and $n$-ary quasigroupoid (Proposition~\ref{symm_wit_existence}); then, we show that there is a finite nontrivial abelian group $G$ with a regular, $(n+2)$-homogeneous action on the quasigroupoid (see Definition~\ref{group_comp} and Lemma~\ref{alt}) and apply Theorem~\ref{associativity_action} to obtain associativity.


\begin{definition}
\label{rel_definable}
If $A \subseteq \C$, then a set $X \subseteq \C $ is \emph{relatively $A$-definable} if it is definable in the expansion $\mathfrak{C}_A$ with a single new unary predicate symbol naming the set $A$.

\end{definition}

\begin{definition}
\label{rel_def_quasigr}
Let $A\subseteq \C$. A quasigroupoid $\mathcal{H}=(P_1,\dots,P_n,Q)$ is \emph{relatively $A$-definable} if the sets $P_i$, $i=1,\dots,n$, the projections $\pi^i$, $i=2,\dots,n$, as well as the relation $Q$ are relatively $A$-definable.
\end{definition}

\begin{theorem}
\label{polygroup_definability}
Let $T$ be a stable theory, $n \geq 2$, $A = \acl(A)$ a small subset of $\C$, and suppose that $T$ has $k$-uniqueness for every $k \in \{2, \ldots, n\}$. Then $T$ does \emph{not} have $(n+1)$-uniqueness over $A$ if and only if there some infinite Morley sequence $I$ over $A$ such that:

\begin{enumerate}
\item There is a relatively $I$-definable $n$-ary polygroupoid $\mathcal{H} = (I, P_2, \ldots, P_n, Q)$ (that is, with $P_1$ equal to the Morley sequence $I$ and each of the sets $P_2, \ldots, P_n$ and $Q$ relatively $I$-definable);
\item $\mathcal{H}$ is locally finite and connected;
\item There is a finite nontrivial $I$-definable group $G$ with a relatively $I$-definable regular action on $\mathcal{H}$;
\item For every two elements $f, g \in P_n$ such that $\pi^n(f) = \pi^n(g)$, if $f, g \in P_n(a_1, \ldots, a_n)$ where $(a_1, \ldots, a_n) \in I^{(n)}$, then there is some $\varphi \in \Aut(\acl_A(a_1, \ldots, a_n) / \bd_A(a_1, \ldots, a_n))$ such that $\varphi(f) = g$.

\end{enumerate}
\end{theorem}

We emphasize that in the hypothesis of Theorem~\ref{polygroup_definability}, we assume full $k$-uniqueness for all $k \leq n$ (that is, for \emph{every} small set $B \subseteq \mathfrak{C}$, that $T$ has $k$-uniqueness over $B$), but that the conclusion of the theorem is ``local'': assuming a failure of $(n+1)$-uniqueness over some base set $A$, we get a relatively $I$-definable $n$-ary polygroupoid for some $I$ which is Morley over the same base set $A$.

The remainder of this section will be a series of lemmas and definitions whose main purpose is to establish Theorem~\ref{polygroup_definability}.

\textbf{We assume throughout the remainder of this section that $n \geq 2$ and that the theory $T$ is stable and has $k$-uniqueness for all $k \in \{2, \ldots, n\}$.} Our object is to study witnesses to the failure of $(n+1)$-uniqueness, which will be measured by the group $G$ defined below.

In the proof below, we denote by $s_i$ the ordered tuple in $[n+1]^{(n)}$ which lists the elements of $[n+1] \setminus \{i\}$ in increasing order.  For distinct $i, j \in [n+1]$, let $s_{i,j} = [n+1] \setminus \{i,j\}$, also listed in the increasing order.

\begin{lemma}
\label{symm_lemma}
(``Symmetrization Lemma'') Suppose that $\{a_1, \ldots, a_{n+1} \}$ is a Morley sequence over $A$. For any set $\{c_i : i\in [n+1]\}$ such that $c_i \in \acl_A(\{a_j : j \in s_i\})$ for each $i \in [n+1]$, there are tuples $\ov{a}_{s_{i,j}}$ enumerating $\acl_A(\{a_k : k \in s_{i,j} \})$ and elements $f_i$ for each $i \in [n+1]$ such that:

\begin{enumerate}
\item $f_i \in \acl_A(\{a_j : j\in s_i\})$;
\item if $\ov{a}_{\bd(s_i)}$ denotes the tuple $(\ov{a}_{s_{i,1}}, \ldots , \ov{a}_{s_{i, i-1}}, \ov{a}_{s_{i, i+1}}, \ldots, \ov{a}_{s_{i, n+1}})$, then for any two $i, j \in [n+1]$, $\tp(f_i, \ov{a}_{\bd(s_i)} / A) =  \tp(f_j, \ov{a}_{\bd(s_j)} / A);$
\item $c_i \in \dcl(f_i)$;
\item for every $i \in [n+1]$, if $c_i \notin \dcl(\ov{a}_{\bd(s_i)})$, then $f_i \notin \dcl(\ov{a}_{\bd(s_i)})$; and
\item if $c_i \in \dcl_A(\{c_j : j \ne i \})$ for every $i \in [n+1]$, then $f_i \in \dcl(\{f_j : j \neq i\})$ for each $i \in [n+1]$.
\end{enumerate} 
\end{lemma}

\begin{proof}

By Lemma~\ref{n_system_existence}, there is an $(n-1)$-symmetric system for $\{a_1, \ldots, a_{n+1}\}$ over $A$. Fix tuples $\ov{a}_s$ for $s \in [n+1]^{(\leq n-1)}$ and elementary maps $\{\varphi_\sigma : \sigma \in \textup{Sym}([n+1]) \}$ as in Definition~\ref{n_system}.

For any two $i, j \in [n+1]$, let $\sigma_{i,j}$ be the unique permutation of $[n+1]$ which sends the tuple $s_i$ to the tuple $s_j$, and for ease of notation let $\varphi_{i,j} := \varphi_{\sigma_{i,j}}$. In particular, note that $\varphi_{i,j}(\ov{a}_{\bd(s_i)}) = \ov{a}_{\bd(s_j)}$, and so $\tp(\ov{a}_{\bd(s_i)} / A) = \tp(\ov{a}_{\bd(s_j)} / A)$.

For each $i \in [n+1]$, let $d_i$ be a finite tuple listing all realizations of $ \tp(c_i / \ov{a}_{\bd(s_i)})$. Let $f_1 \in \acl_A(\{a_j : j \neq 1\})$ be a finite tuple listing all of the elements of $$\{ \varphi_\sigma(d_j) : \sigma \textup{ is a permutation of } [n+1] \textup{ such that } \sigma(j) = 1 \},$$ and for every $i \in \{2, \ldots, k+1\}$, let $f_i = \varphi_{\sigma_{1,i}}(f_1)$. This immediately implies that $f_i \in \acl_A(\{a_j : j \neq i\})$ and $$\tp(f_1, \ov{a}_{\bd(s_1)} / A) = \tp(f_i, \ov{a}_{\bd(s_i)} / A),$$ so properties (1) and (2) of the Lemma hold.


Before proving (3), (4), and (5), we establish some preliminary claims. For the rest of the proof of the Lemma, we will assume that for every $i \in [n+1]$, $c_i \in \dcl_A(\{c_j : j \neq i\}).$

\begin{claim}
\label{closure_types}
If $e \in f_i$ and $\tp(e' / \ov{a}_{\bd(s_i)}) = \tp(e / \ov{a}_{\bd(s_i)})$, then $e' \in f_i$.
\end{claim}

\begin{proof}
Since $\tp(f_1, \ov{a}_{\bd(s_1)} / A) = \tp(f_i, \ov{a}_{\bd(s_i)} / A),$ it suffices to prove the Claim for $f_1$. But this is clear, since $f_1$ comprises of tuples enumerating all the realizations of certain types over $\ov{a}_{\bd(s_1)}$.
\end{proof}

It is worth emphasizing that $\varphi_{\sigma} \circ \varphi_{\tau}$ is not necessarily equal to $\varphi_{\sigma \circ \tau}$, so there is an asymmetry in our definitions: $f_1$ is defined differently than the other $f_i$'s. On the other hand, since for every $i$ we have $$\varphi_{\sigma \circ \tau} (\ov{a}_{\bd(s_i)}) = \varphi_\sigma \left( \varphi_\tau(\ov{a}_{\bd(s_i)}) \right),$$ we immediately obtain:

\begin{claim}
\label{phi_sigma_maps}
For any two permutations $\sigma, \tau$ of $[n+1]$ and any $d \in \acl_A(\{a_j : j \neq i\})$, if $\sigma(\tau(i)) = k$, then $$\tp( \varphi_{\sigma \circ \tau}(d), \ov{a}_{\bd(s_k)} / A) = \tp( \varphi_\sigma ( \varphi_\tau ( d)), \ov{a}_{\bd(s_k)} / A).$$
\end{claim}

\begin{claim}
\label{d_i}
For every $i \in [n+1]$, $d_i \in f_i$.
\end{claim}

\begin{proof}
By applying the elementary map $\varphi_{i,1}$, we obtain $$\tp(d_i, f_i, \ov{a}_{\bd(s_i)} / A) = \tp(\varphi_{i,1}(d_i), f_1, \ov{a}_{\bd(s_1)} / A).$$ But $\varphi_{i,1}(d_i) \in f_1$ by the very definition of $f_1$, and therefore $d_i \in \dcl(f_i)$.

\end{proof}

By the last Claim above, $c_i \in d_i \in f_i$, establishing (3). Property (4) follows immediately from (3).

\begin{claim}
\label{interdef_d}
For every $i \in [n+1]$, $d_i \in \dcl(\{d_j : j \neq i\})$.
\end{claim}

\begin{proof}
Fix some $i \in [n+1]$ and some $c'_i \in d_i$ such that $\tp(c'_i / \ov{a}_{\bd(s_i)}) = \tp(c_i / \ov{a}_{\bd(s_i)})$. Pick any other $j \in [n+1] \setminus \{i\}$, and we claim that there is an element $c'_j$ of $d_j$ (that is, $\tp(c'_j / \ov{a}_{\bd(s_j)}) = \tp(c_j / \ov{a}_{\bd(s_j)})$) such that $c'_i \in \dcl (\{c'_j \} \cup \{c_k : k \neq i, j \})$. (Apply relative $(n,n)$-uniqueness over the base set $\acl(a_j)$ to construct an elementary map $\varphi$ which fixes every tuple $\ov{a}_{\bd(s_k)}$ pointwise, fixes every $c_k$ such that $k \neq i, j$, and sends $c_i$ to $c'_i$; then let $c'_j$ be $\varphi(c_j)$.)
\end{proof}

\begin{claim}
\label{phi_map_preservation}
For any permutation $\sigma$ of $[n+1]$ and any element $e$ of the tuple $f_i$, $\varphi_\sigma(e)$ is an element of the tuple $f_{\sigma(i)}$.
\end{claim}

\begin{proof}
Without loss of generality, we consider an element $e \in f_i$ such that for some $k \in [n+1]$ and some permutation $\tau$ of $[n+1]$ such that $\tau(k) = 1$, $$\tp(e, \ov{a}_{\bd(s_i)} / A) = \tp(\varphi_\tau(d_k), \ov{a}_{\bd(s_1)} / A).$$ Applying the elementary map $\varphi_{1,i}$ plus Claim~\ref{phi_sigma_maps}, we conclude $$\tp(e, \ov{a}_{\bd(s_i)} / A) = \tp( \varphi_{\sigma_{1, i} \circ \tau} (d_k), \ov{a}_{\bd(s_i)} / A).$$ Now fix any permutation $\sigma$ of $[n+1]$ and say $\sigma(i) = j$. Then $$\tp(e, \ov{a}_{\sigma^{-1}(\bd(s_j))} / A) = \tp( \varphi_{\sigma_{1, i} \circ \tau} (d_k), \ov{a}_{\sigma^{-1}(\bd(s_j))} / A).$$ (This is because $\ov{a}_{\sigma^{-1}(\bd(s_j))}$ is a permutation of $\ov{a}_{\bd(s_i)}$ and hence any elementary map $\psi$ such that $\psi(e) =  \varphi_{\sigma_{1, i} \circ \tau} (d_k)$ and $\psi$ fixes $\ov{a}_{\bd(s_i)} \cup A$ pointwise will also witness the equality of the two displayed types above.) Applying the elementary map $\varphi_\sigma$ to the two types above and citing Claim~\ref{phi_sigma_maps}, we obtain $$\tp(\varphi_\sigma(e), \ov{a}_{\bd(s_j)} / A) = \tp(\varphi_{\sigma \circ \sigma_{1,i} \circ \tau}(d_k), \ov{a}_{\bd(s_j)} / A) = \tp(\varphi_{\sigma_{j,1} \circ \sigma \circ \sigma_{1,i} \circ \tau}(d_k), \ov{a}_{\bd(s_1)} / A).$$ But by the definitions of $f_1$ and of $f_j$, the element $\varphi_{\sigma_{j,1} \circ \sigma \circ \sigma_{1,i} \circ \tau}(d_k)$ is in $f_1$ and hence the equality of types displayed above implies that $\varphi_\sigma(e) \in f_j$.

\end{proof}



\begin{claim}
\label{f_1}
Any element $\varphi_\sigma(d_j)$ of $f_1$ is in $\dcl(f_2, \ldots, f_{n+1})$.
\end{claim}

\begin{proof}
By Claim~\ref{interdef_d}, $d_j \in \dcl(\{d_k : k \neq j\})$. By Claim~\ref{d_i}, $d_k \in f_k$, so $d_j \in \dcl(\{f_k : k \neq j\})$. This definability relation is clearly preserved under application of the elementary map $\varphi_\sigma$, and so Claim~\ref{f_1} now follows by Claim~\ref{phi_map_preservation}.

\end{proof}

Finally, we show that for any $i \in [n+1]$, $f_i \in \dcl(\{f_j : j \neq i\})$. Fix some $e \in f_i$ such that $$\tp(e, \ov{a}_{\bd(s_i)} / A) = \tp(\varphi_\sigma(d_k), \ov{a}_{\bd(s_1)} / A) $$ for some permutation $\sigma$ of $[n+1]$. Therefore $$\tp(e, \ov{a}_{\bd(s_i)} / A) = \tp(\varphi_{1,i} ( \varphi_\sigma(d_k)), \ov{a}_{\bd(s_i)} / A),$$ and applying the same argument by relative $(n,n)$ uniqueness as in the proof of Claim~\ref{interdef_d}, it suffices to show that $\varphi_{1,i} ( \varphi_\sigma(d_k)) \in \dcl(\{f_j : j \neq i\})$.

By Claim~\ref{f_1}, $\varphi_\sigma(d_k) \in \dcl(f_2, \ldots, f_{n+1})$. Applying the elementary map $\varphi_{1,i}$ and Claim~\ref{phi_map_preservation}, we conclude that $\varphi_{1,i} (\varphi_{\sigma}(d_k)) \in \dcl(\{ f_j : j \neq i \})$, and by the previous paragraph, we are finished with the proof of the Lemma.

\end{proof}

\begin{proposition}
\label{Morley_seq_witness}
Suppose that $T$ is stable and has $k$-uniqueness for every $k \leq n$ but fails to have $(n+1)$-uniqueness. Then there is a small algebraically closed set $B$, a Morley sequence $a_1\dots a_{n+1}$ over $B$, and elements $$c_i \in \acl_B(\{a_1, \ldots, a_{n+1}\} \setminus \{a_i\})$$ such that $$c_{n+1} \in \dcl_B( \{c_1, \ldots, c_n\}) \setminus \bd_B(a_1, \ldots, a_n).$$
\end{proposition}

\begin{proof}
Since $(n+1)$-uniqueness fails and $\le n$-uniqueness holds, by Fact~\ref{Bn_uniqueness} above, we conclude that the property $B(n+1)$ must also fail. We use the equivalent formulation of $B(n+1)$ given by Fact~\ref{Bn_equivalent}. 

Fix some $B$-independent set $\{b_1, \ldots, b_{n+1}\}$ and a map $\varphi$ as in Fact~\ref{Bn_equivalent}, and assume without loss of generality that the base set $B = \emptyset$ (by adding constants to the language if necessary): namely, we have $\varphi \in \Aut(\ov{B}_{[n]} / \bd(b_1, \ldots, b_n))$ cannot be extended to an automorphism that fixes  $$\ov{B}_{[n+1] \setminus \{1\}} \cup \ldots \cup \ov{B}_{[n+1] \setminus \{n\}}$$ pointwise. . We will build a Morley sequence $\{a_1,\dots,a_{n+1}\}$ and an automorphism $\varphi' \in \Aut(\ov{A}_{[n]} / \bd(a_1, \ldots, a_n))$ which cannot be extended to an automorphism that fixes $\ov{A}_{[n+1] \setminus \{1\}} \cup \ldots \cup \ov{A}_{[n+1] \setminus \{n\}}$ pointwise.

Let $a_1$ be a tuple $a_1^1\dots a_1^{n+1}$ such that $q:=\tp(a_1^1\dots a_1^{n+1})=\tp(b_1 \ldots b_{n+1})$. Take a set $\{a_2,\dots,a_{n+1}\}$ of realizations of the type $q$ such that the set $\{a_1,\dots,a_{n+1}\}$ is independent. By stationarity, the type of the sequence $(a^1_1,a^2_2,\dots,a^{n+1}_{n+1})$ is also
$q$. Thus, we may assume that $(a^1_1,\dots,a^{n+1}_{n+1}) = (b_1\dots,b_{n+1})$.


It remains to build the map $\varphi' \in \Aut(\ov{A}_{[n]})$ using relative $(n,n)$-uniqueness by amalgamating the map $\varphi$ and the identity automorphisms on the algebraic closures of tuples $a_1,\dots,\widehat{a_i},\dots,a_n$, where $i=1,\dots, n$. 
The construction of $\varphi'$ has two stages. For the first stage, we perform a series of amalgamations to build a system of maps $\langle \varphi''_s : s \subseteq \{1, \ldots, n\} \rangle$ such that: 

\begin{enumerate}
\item If $s \subseteq \{1, \ldots, n-1\}$, then $\varphi''_s$ is the identity map on $\overline{a}_s$;
\item If $s = s_0 \cup \{n\}$ where $s_0 \subseteq \{1, \ldots, n-1\}$, then $\varphi''_s$ is an elementary permutation of $\acl(\overline{a}_{s_0} \cup b_n)$ which extends $\varphi$; and
\item If $s \subseteq t$, then $\varphi''_s \subseteq \varphi''_t$.
\end{enumerate}

The construction of the maps $\varphi''_s$ for $|s| > 1$ goes by induction on $|s|$, always amalgamating over the base set $\acl(b_1, \ldots, b_{n-1})$ and considering $\{a_1, \ldots, a_{n-1}, b_n\}$ to be the set of independent ``vertices:'' if we consider $s = \{s_1, \ldots, s_k\} \subseteq \{1, \ldots, n\}$, then once we have the $\varphi''_t$ for all $|t| < |s|$, construct $\varphi''_s$ using relative $(k,k)$-uniqueness to find a common extension of the maps $\langle \varphi''_t : |t| < k \rangle$.

Then let $\varphi'' := \varphi''_{\{1, \ldots, n\}} \in \Aut(\acl(a_1,\dots,a_{n-1},b_{n}))$. Note that $\varphi'' \restriction \acl(b_1,\dots,b_{n}) = \varphi$ and for all
$i\in 1,\dots, n-1$ we have $\varphi''\restriction \acl(a_1,\dots,\widehat{a_i},\dots,a_{n-1},b_{n})$ is an identity automorphism.

Now we obtain $\varphi'$ using relative $(n,n)$-uniqueness over $\acl(b_n)$ on the system of $n$ compatible automorphisms given by $\varphi''$ plus the $(n-1)$ the identity automorphisms on $\ov{A}_{[n] \setminus \{i\}}$ for $i=1,\dots, n-1$.

Finally, we note that it is impossible to amalgamate $\varphi'$ with the identity automorphisms on $\ov{A}_{[n+1] \setminus \{i\}}$ for $i=1,\dots,n$ simply because $\varphi \subseteq \varphi'$ and we assumed that $\varphi$ could not be extended to an elementary map fixing $\ov{B}_{[n+1] \setminus \{1\}} \cup \ldots \cup \ov{B}_{[n+1] \setminus \{n\}}$ pointwise.

Now from the existence of the elementary map $\varphi'$, we quickly find the desired elements $c_i$: the impossibility of amalgamating $\varphi'$ with identity maps on the other ``edges'' implies that there are finite tuples $d_i \in \ov{A}_{[n+1] \setminus \{i\}}$ such that $\varphi'(d_{n+1}) \neq d_{n+1}$ and $$\tp(\varphi'(d_{n+1}) / d_1, \ldots, d_n) \neq \tp(d_{n+1} / d_1, \ldots, d_n),$$ and we let $c_{n+1}$ be a code for the finite set $X$ of all realizations of $\tp(d_{n+1} / d_1, \ldots, d_n)$ and let $c_i = d_i$ for every $i \in \{1, \ldots, n\}$; then evidently $c_{n+1} \in \acl(a_1, \ldots, a_n)$ and $c_{n+1} \in \dcl(c_1, \ldots, c_{n+1})$, but since $\varphi'(X) \neq X$ and $\varphi' $ fixes $\bd(a_1, \ldots, a_n)$ pointwise, $c_{n+1} \notin \bd(a_1, \ldots, a_n)$. 

\end{proof}

\begin{definition}
\label{symm_wit}
A \emph{symmetric witness to the failure of $(n+1)$-uniqueness (over $B = \acl(B)$)} is an $(n+1)$-element Morley sequence $I$ over $B$ and a relatively $I$-definable connected quasigroupoid $(I,P_2,\dots,P_n,Q)$ such that

\begin{enumerate}
\item (Isolation of types)
\begin{enumerate}
\item For any $i \in \{2, \ldots, n\}$ and any $f \in P_i$, the type $\tp(f / \bd(\supp(f)))$ is isolated over $\pi^i(f)$.
\item For any two elements $f, g$ of $P_n$ such that $\pi^n(f) = \pi^n(g)$, the type $\tp(f,g / \bd(\pi(f)))$ is isolated over $\pi^n(f)$.
\end{enumerate}
\item (Algebraicity) For any $i \in \{2, \ldots, n\}$ and any $f \in P_i$, $f \in \acl(\pi(f))$.

\end{enumerate}
\end{definition}

For $n=3$, the definition above is almost the same as that of a ``full symmetric witness'' from the article \cite{gkk} (see also the older Definition~2.2 of \cite{GK}), but for the purposes of the current paper we found it convenient to also assume that the types of all \emph{pairs} from $P_n$ are isolated over their projections; this condition is easy to enforce, and will help in establishing the definability of the group $G$ (see Definitions~\ref{G_notation} and \ref{group_comp} below).

\begin{proposition}
\label{symm_wit_existence}
If $T$ is stable with $k$-uniqueness for every $k \leq n$ but which does not have $(n+1)$-uniqueness, then $T$ has a symmetric failure to $(n+1)$-uniqueness.
\end{proposition}

\begin{proof}

By Proposition~\ref{Morley_seq_witness}, there is a Morley sequence $a_1\dots a_{n+1}$ over some small set $B$ and elements $c_i \in \acl_B(\{a_1, \ldots, a_{n+1}\} \setminus \{a_i\})$ such that $$c_{n+1} \in \dcl_B( \{c_1, \ldots, c_n\}) \setminus \bd_B(a_1, \ldots, a_n).$$ For simplicity of notation, we will assume in the proof below that $B = \emptyset$ (expanding the language to add constants for the elements of $B$ if necessary).

Without loss of generality, $a_j \in c_j$ whenever $i, j$ are two distinct elements of $[n+1]$. We claim that we may assume that not only $c_{n+1} \in \dcl(c_1, \ldots, c_n)$, but also that $c_i \in \dcl(\{c_j : j \neq i \})$ for each $i \in [n+1]$. This can be acheived by replacing the original elements $\{c_1, \ldots, c_n\}$ by elements $c'_i$ chosen recursively on $i$ so that $c'_i \in T^{eq}$ is a canonical name for the finite set of all realizations of $$\tp(c_i / c'_1, \ldots, c'_{i-1}, c_{i+1}, \ldots, c_{n+1})$$ (which is clearly algebraic, since $c_i$ is in the algebraic closure of $\{a_j : j \neq i \}$ and by the assumption that $c_k$ and $c'_k$ contain $a_j$ if $k \neq j$). Then an easy induction argument shows that the new elements $\{c'_1, \ldots, c'_n, c_{n+1} \}$ are interdefinable as we wish.

We apply Lemma~\ref{n_system_existence} to construct an $(n-1)$-symmetric system $\{\ov{a}_s : s \in [n+1]^{(< n)} \}$. As above, $s_{i,j}$ denotes the ordered $(n-1)$-tuple which lists the elements of $[n+1] \setminus \{i,j\}$ in increasing order and $\ov{a}_{\bd(s_i)}$ is the tuple $(\ov{a}_{s_{1,i}}, \ldots, \ov{a}_{s_{n+1, i}})$. Throughout the proof of this Lemma, we use the phrase $s \in [n+1]^{(k)}$ to mean (by abuse of notation) that $s$ is an ordered tuple listing $k$ distinct elements of $[n+1]$ in increasing order.



\textbf{Step 1:} By the Symmetrization Lemma (Lemma~\ref{symm_lemma}) applied to the elements $c_i$, there is a tuple of elements $(f^n_1, \ldots, f^n_{n+1})$ satisfying all of the following properties:

\begin{enumerate}
\item $f^n_i \in \acl(\{a_j : j \neq i\} )\setminus \bd(a_1, \ldots, \widehat{a}_i, \ldots, a_{n+1})$;
\item For each $i \in \{1, \ldots, n+1\}$, $f^n_i \in \dcl(f^n_1, \ldots, \widehat{f}^n_i, \ldots, f^n_{n+1})$; and
\item For any two $i, j \in \{1, \ldots, n+1\}$, $$(f^n_i, \ov{a}_{\bd(s_i)}) \equiv (f^n_j, \ov{a}_{\bd(s_j)}).$$
\end{enumerate}

\textbf{Step 2:} Next, for each $m \in \{1, \ldots, n-1\}$ and every $s \in [n+1]^{(m)}$, we construct elements $f^m_s$ satisfying both of the following properties (letting $f^n_i = f^n_{s_i}$ for uniformity of notation):

\begin{enumerate}
\setcounter{enumi}{3}
\item Whenever $s = (i_1, \ldots, i_k) \in [n+1]^{(k)}$ is increasing and $2 \leq k \leq n$, $\tp(f^{k}_s / \ov{a}_{\bd(s)})$ is isolated over $(f^{k-1}_{s \setminus \{i_1\}}, \ldots, f^{k-1}_{s \setminus \{i_k\}})$; 
\item If $m = n-1$ and $s = (i_1, \ldots, i_n) \in [n+1]^{(n)}$, then for any $g^n_s$ such that $$(g^n_s, \overline{a}_{\bd(s)}) \equiv (f^n_s, \overline{a}_{\bd(s)}),$$ we have that $\tp(f^n_s, g^n_s / \overline{a}_{\bd(s)})$ is isolated over $(f^m_{s \setminus \{i_1\}}, \ldots, f^m_{s \setminus \{i_n\}})$; and
\item If $|s| = |s''| = k$, then $(f^k_s, \overline{a}_{\bd(s)}) \equiv (f^k_{s'}, \overline{a}_{\bd(s')})$.
\end{enumerate}

For Step~$2$, we perform the construction of the $f^{n-\ell}_s$ by induction on $\ell$.

Given some $m \in \{2, \ldots, n\}$, suppose that we are given elements $f^m_s$ satisfying (4), (5), and (6) for every increasing $s \in [n+1]^{(m)}$.  Consider $s = (1, \ldots, m)$ and the element $f^m_s$. Pick elements $c_1, \ldots, c_m$ such that $c_i \in \bd(a_1, \ldots, \widehat{a}_i, \ldots, a_m)$ and $\tp(f^m_s / \ov{a}_{\bd(s)})$ is isolated over $(c_1, \ldots, c_m)$. In the special case where $m = n$, we can also pick the $c_i$ such that whenever $(g^n_s, \overline{a}_{\bd(s)}) \equiv (f^n_s, \overline{a}_{\bd(s)})$, the type $\tp(f^n_s, g^n_s / \overline{a}_{\bd(s)})$ is also isolated over $c_1, \ldots, c_n$ (since there are only finitely many such $g^n_s$). Now apply the Symmetrization Lemma (Lemma~\ref{symm_lemma}) to the elements $c_i$ to obtain elements $f^{m-1}_{s \setminus \{1\}}, \ldots, f^{m-1}_{s \setminus \{m\}}$ such that for any $i, j \in [m]$, $$(f^{m-1}_{s \setminus \{i\}} , \ov{a}_{s \setminus \{i\}}) \equiv (f^{m-1}_{s \setminus \{j\}} , \ov{a}_{s \setminus \{j\}}).$$ Because we also have that $c_i \in f^{m-1}_{s \setminus \{i\}}$, the isolation conditions (4) and (5) clearly hold for these elements. Finally, if $t \in [n+1]^{(m-1)}$ is any increasing tuple which is not contained in $[m]$, we let $f^{m-1}_t $ be some element such that $$(f^{m-1}_t,  \ov{a}_{\bd(t)}) \equiv (f^{m-1}_{(1, \ldots, m-1)}, \ov{a}_{\bd((a_1, \ldots, a_{m-1}))}).$$



Finally, we explain how to put the pieces together an use Steps~$1$ and $2$ to build a symmetric witness to the failure of $(n+1)$-uniqueness as in Definition~\ref{symm_wit}. For every nonempty tuple $s \in [n+1]^{(\leq n)}$, we construct a tuple $g_s$ by a recursion on $|s|$: first, the tuples $g_{(i)}$ are defined to be $f^1_{(i)}$ as above; then, for the recursion, if $s = (i_1, \ldots, i_m)$, let $$g_s = (f^{m}_s, g_{s \setminus \{i_1\}}, \ldots, g_{s \setminus \{i_m\}}).$$ Note that each member $a_i$ of the original Morley sequence is interalgebraic with $g_{(i)}$, and we let $I$ be an infinite Morley sequence in the type of $g_{(1)}$. For $i \in \{2, \ldots, n\}$, let $P^0_i(x; y_1, \ldots, y_i)$ be a formula such that $P^0_i(x; g_{(1)}, \ldots, g_{(i)})$ isolates the type of $g_{(1, \ldots, i)}$ over $(g_{(1)}, \ldots, g_{(i)})$ (this type is clearly atomic since it is algebraic), and let $P_i$ be the relatively $I$-definable set $$P_i := \{b \in \mathfrak{C} : P^0_i(b, c_1, \ldots, c_i) \textup{ for some tuple of pairwise distinct elements } c_1, \ldots, c_i \textup{ of } I\}.$$ F for each $i \in \{2, \ldots, n\}$, there is a definable map $\pi^i : P^i \rightarrow (P^{i-1})^{i}$ which sends $g_{(1, \ldots, i)}$ to the tuple $(g_{(2, \ldots, i)}, \ldots, g_{(1, \ldots, i-1)})$: in fact, $\pi^i(z)$ simply sends $z$ to the appropriate subtuple of $z$. The compatibility of the tuple $\pi^i(f)$ in the sense of Definition~\ref{comp_tuples} for any $f \in P^i$ is straightforward to check. 

Note that for any $i \in [n+1]$, the element $g_{s_i}$ is in $\dcl(\{g_{s_j} : j \neq i\})$: this is because $g_{s_i}$ consists of elements $g_{s_{i,j}}$ which are definitionally included in the other elements $g_{s_j}$, plus the element $f^n_{s_i}$, which is in the definable closure of the elements $f^n_{s_j} \in g_{s_j}$ for by Condition~(2) of Step~1. Let $Q(x_1, \ldots, x_{n+1})$ be a formula in the type of $(g_{s_1}, \ldots, g_{s_{n+1}})$ witnessing that the tuple is compatible and that any $g_{s_i}$ is definable from all of the other $g_{s_j}$'s (so that axioms (3) and (4) of Definition~\ref{quasigroupoid} are satisfied).

The property of isolation of types (condition~(1) of Definition~\ref{symm_wit}) holds because of parts~(4) and (5) of Step 2 above.

The last thing to check is that the $n$-ary quasigroupoid $(I, P_2, \ldots, P_n, Q)$ is connected. So suppose that $(p_1, \ldots, p_{n+1})$ is any compatible $(n+1)$-tuple from $P_n$. Say $(c_1, \ldots, c_{n+1})$ is the tuple from $I^{(n+1)}$ such that $\supp(p_i) = (c_1, \ldots, \widehat{c}_i, \ldots, c_{n+1})$. Since $I$ is an indiscernible set, the definition of $P_n$ gives us that for each $i$, $$(p_i, c_1, \ldots, \widehat{c}_i, \ldots, c_{n+1}) \equiv (g_{s_i}, g_{(1)}, \ldots, \widehat{g}_{(i)}, \ldots, g_{(n+1)}).$$ Now if we fix any $i \in [n+1]$,  we can use $(\leq n)$-uniqueness over the set $\acl(c_{i})$ to conclude that $$(p_1, \ldots, \widehat{p}_i, \ldots, p_{n+1}) \equiv (g_{s_1}, \ldots, \widehat{g}_{s_i}, \ldots, g_{s_{n+1}}),$$ and in particular $\exists x \left[ Q(p_1, \ldots, p_{i-1}, x, p_{i+1}, \ldots, p_{n+1}) \right]$, as required for connectedness.


\end{proof}

From this point on in the section, we fix a Morley sequence $I$ over $\emptyset = \acl^{eq}(\emptyset)$ and a relatively $I$-definable connected $n$-ary quasigroupoid $(I, P_2 \ldots, P_n, \pi^k : 2 \leq k \leq n, Q)$ which is a symmetric witness to the failure of $(n+1)$-uniqueness over $\emptyset$. Our next task is to construct the group $G$ which will act regularly on this $n$-ary quasigroupoid.

\begin{definition}
\label{symm_wit_notation}
Within the context of the fixed relatively $I$-definable $n$-ary quasigroupoid above:
\begin{enumerate}
\item If $i \in \{2, \ldots, n\}$ and $(f_1, \ldots, f_i)$ is a compatible tuple of elements of $P_{i-1}$, then $P_i(f_1, \ldots, f_i)$ denotes the set of all elements $f$ of $P_i$ such that $\pi^i(f) = (f_1, \ldots, f_i)$. 
\item If $(a_1, \ldots, a_i) \in I^{(i)}$, then $P_i(a_1, \ldots, a_i)$ is the set of all $f \in P_i$ which are over $(a_1, \ldots, a_i)$ (in the sense of Definition~\ref{over_P1}). If $f \in P_i(a_1, \ldots, a_i)$, then we define $\pi(f) := (a_1, \ldots, a_i)$.
\end{enumerate}

Note that due to algebraicity and isolation of types (clauses~(1) and (2) of Definition~\ref{symm_wit}), the sets $P_i(\overline{f})$ and $P_i(\overline{a})$ in (2) and (3) above are always finite.
\end{definition}



\begin{lemma}
\label{fiber_actions}
Fix some $(a_1, \ldots, a_i, \ldots, a_{n+1}) \in I^{(n+1)}$, two distinct numbers $i, j \in \{1, \ldots, n+1\}$, and two compatible $n$-tuples $\overline{f}, \overline{f}'$ from $P_{n-1}$ such that $$\overline{f} = \langle f_k : 1 \leq k \leq n+1; k \neq i \rangle,$$ $$\overline{f}' = \langle f'_k : 1 \leq k \leq n+1 ; k \neq j \rangle,$$ $$\pi(f_k) = (a_1, \ldots, a_{n+1}) \setminus \{a_i, a_k\},$$ $$\pi(f'_k) = (a_1, \ldots, a_{n+1}) \setminus \{ a_j, a_k\},$$ and $$f_j = f'_i.$$ Then given any pair $(p,q)$ from $P_n(\overline{f}) \times P_n(\overline{f})$, we can define a permutation $\chi = \chi(i,j; p,q, \overline{f}')$ of $P_n(\overline{f}')$ by the rule: whenever $(r_1, \ldots, r_{n+1})$ is a compatible tuple from $P_n$ such that $$Q(r_1, \ldots, r_{i-1}, p, r_{i+1}, \ldots, r_j, \ldots, r_{n+1})$$ holds, then $$ Q(r_1, \ldots, r_{i-1}, q, r_{i+1}, \ldots, \chi(r_j), \ldots, r_{n+1})$$ holds as well.

\end{lemma}

\begin{proof}
Most of the work will go into showing that this rule gives a well-defined function $\chi$. For ease of notation, let $i = 1$ and $j = n+1$ (the other cases have identical proofs).

What we must show is: given any elements $r_2, \ldots, r_{n+1}$, $r'_2, \ldots, r'_n,$ $s_{n+1}$, and $s'_{n+1}$ such that

$$Q(p, r_2, \ldots, r_n, r_{n+1}) \wedge Q(q, r_2, \ldots, r_n, s_{n+1}),$$

$$Q(p, r'_2, \ldots, r'_n, r_{n+1}) \wedge Q(q, r'_2, \ldots, r'_n, s'_{n+1}),$$

and the four $(n+1)$-tuples in the $Q$ relations above are all compatible and independent, then $s'_{n+1} = s_{n+1}$.

\begin{claim}
There is a family of elementary maps $$\langle \varphi_s : \{2\} \subseteq s \subseteq \{1, \ldots, n+1 \} \rangle$$ satisfying all of the following properties:

\begin{enumerate}
\item $\varphi_s \in \Aut(\overline{A}_{s})$;
\item If $s \subseteq t$ then $\varphi_s \subseteq \varphi_t$;
\item If $1 \notin s$ or if $n+1 \notin s$ then $\varphi_s$ is the identity map on $\overline{A}_s$;
\item If $\{1, 2, n+1\} \subseteq s$ and $i \in \{1, \ldots, n+1\} \setminus s$, then $$\varphi_s(\pi_{s }(r_i)) = \pi_{s }(r'_i),$$ where the notation ``$\pi_s(r_i)$'' stands for the natural projection of $r_i$ onto an element of $P_{|s|}( \langle a_i : i \in s \rangle )$ formed by composing various maps $\pi^k_\ell$; and
\item If $i \in \{3, \ldots, n\}$ and $s_i = \{1, \ldots, n+1\} \setminus \{ i\}$, then $\varphi_{s_i}(r_i) = r'_i$.
\end{enumerate}

\end{claim}

For the proof of the Claim, we construct the maps $\varphi_s$ by induction on $|s|$ using amalgamation properties. For the base case, where $s = \{2\}$, we simply let $\varphi_{\{2\}}$ be the identity map on $\acl(a_2)$. For the induction step, suppose that $|s| = k > 1$ and we are given the maps $\varphi_t$ for all $t \subset_{k-1} \{1, \ldots, n+1\}$ such that $2 \in t$. Let $s = \{2\} \cup \{i_1, \ldots, i_{k-1}\}$ and for every $j \in [k-1]$, let $s_j = s \setminus \{i_j\}$. By the induction hypothesis, the functions $\{ \varphi_{s_j} : 1 \leq j \leq k-1\}$ form a compatible system of elementary maps, so by relative $(k-1, n)$-uniqueness over $\acl(a_2)$, the union of the $\varphi_{s_j}$ maps can be extended to an elementary map $\varphi^0_s$ of $\ov{A}_{s}$. 

If $1 \notin s$, then clearly $1 \notin s_j$ for every $j$, so by the induction hypothesis the $\varphi_{s_j}$ are all identity maps, and in this case we may choose $\varphi^0_s = \varphi_s$ to be an identity map as well. The same applies if $n+1 \notin s$, so we have verified condition~(3).

Finally, if $\{1, 2, n+1\} \subseteq s$, then we only need to construct $\varphi_s$ satisfying conditions (4) and (5) for the particular $s$ under consideration. But since $\varphi^0_s$ was constructed as an extension of the maps $\varphi_{s_j}$, we may always construct $\varphi_s$ as $\varphi^1_s \circ \varphi^0_s$ for some $\varphi^1_s \in \Aut(\ov{A}_s)$ which fixes the images of the maps $\left\{ \varphi_{s_j} : j \in [k-1] \right\}$ pointwise.



Now that we have established the Claim, let $\varphi := \varphi_{[n+1]}$ be the union of all of the maps $\varphi_s$. By construction, $\varphi(p) = p$ (since $\varphi$ is the identity on $\ov{A}_{[n+1] \setminus \{1\}}$), $\varphi(r_i) = r'_i$ for every $i \in \{3, \ldots, n\}$, and $\varphi(r_{n+1}) = r_{n+1}$. So because $\varphi$ must preserve the $Q$ relation on $(p, r_2, \ldots, r_n, r_{n+1})$, it follows that $\varphi(r_2) = r'_2$. Therefore since $Q(q, r_2, \ldots, r_n, s_{n+1})$ holds, we also have (by applying $\varphi$) $$Q(q, r'_2, \ldots, r'_n, s_{n+1}),$$ and comparing this with the $Q$ relation on $(q, r'_2, \ldots, r'_n, s'_{n+1})$, we conclude that $s'_{n+1} = s_{n+1}$ as desired.

\end{proof}

\begin{lemma}
\label{f_prime_independence}
Suppose that $i$ and $j$ are distinct elements of $\{1, \ldots, n+1\}$ and we have three compatible $n$-tuples from $P_{n-1}$, $$\overline{f} = \langle f_k : 1 \leq k \leq n+1; k \neq i \rangle,$$ $$\overline{f}' = \langle f'_k : 1 \leq k \leq n+1; k \neq j \rangle,$$ and $$\overline{f}'' = \langle f_k : 1 \leq k \leq n+1; k \neq j \rangle$$ such that $f_j = f'_i = f''_i$. Then for any two pairs $(p,q)$, $(p',q')$ from $P(\overline{f}) \times P(\overline{f})$, $$\chi(i,j;p,q,\overline{f}') = \chi(i,j;p',q',\overline{f}')$$ if and only if $$\chi(i,j;p,q,\overline{f}'') = \chi(i,j;p',q',\overline{f}'').$$
\end{lemma}

\begin{proof}
To fix notation, say that (as in Lemma~\ref{fiber_actions}) $$\pi(f_k) = (a_1, \ldots, a_{n+1}) \setminus \{a_i, a_k\}$$ and $$\pi(f'_k) = (a_1, \ldots, a_{n+1}) \setminus \{ a_j, a_k\}$$ for some $(a_1, \ldots, a_{n+1}) \in I^{(n+1)}$. The first reduction is to note that without loss of generality, $$\pi(f''_k) = (a_1, \ldots, a_{n+1}) \setminus \{ a_j, a_k\},$$ since by $2$-uniqueness whichever new realization of $\tp(a_1)$ occurs in $\pi^n(\overline{f}'')$ can be mapped to $a_i$ via an elementary map which fixes $\acl(\{a_1, \ldots, a_{n+1}\} \setminus \{a_i\})$ pointwise.

The remainder of the proof uses the same argument as in Lemma~\ref{fiber_actions}: if, say, $i=1$ and $j = n+1$, we can perform a similar series of amalgamations to get a compatible family of maps $\langle \varphi_s : \{2\} \subseteq s \subseteq \{1, \ldots, n+1\} \rangle$ which fix the ``face'' $\acl(a_2, \ldots, a_{n+1})$ pointwise and map $\overline{f}'$ to $\overline{f}''$.
\end{proof}

\begin{lemma}
\label{pair_types}
Suppose that $\overline{f} = (f_1, \ldots, f_n)$ is a compatible $n$-tuple from $P_{n-1}$ and $\{r_1, \ldots, r_N\}$ is a list of all the elements of $P_n(\overline{f})$. Then $\tp(r_1 r_i / \overline{f}) = \tp(r_1 r_j / \overline{f})$ if and only if $i = j$.

Thus $$\{\tp(r_i r_j / \bd (\overline{f})) : 1 \leq i, j \leq N\}$$ contains exactly $N$ types, represented by the pairs $(r_1, r_1)$, $(r_1, r_2), \ldots,$ $(r_1, r_N)$.
\end{lemma}

\begin{proof}
If we fix any other compatible tuple $\overline{f}'$ as in the statement of Lemma~\ref{fiber_actions}, then it is immediate from the definition that for any two distinct $i, j \in \{1, \ldots, n+1\}$, if $r_k \neq r_\ell$ then $$\chi(i,j; r_1, r_k, \overline{f}') \neq \chi(i,j;r_1, r_\ell, \overline{f}').$$ The relation above is independent of the particular choice of $\overline{f}'$ by Lemma~\ref{f_prime_independence}, and so this relation is clearly definable over $\acl(\pi(r_1))$ by definability of types. Since $\tp(r_1 r_i / \acl(\pi(r_1)))$ is isolated over $\overline{f}$, this proves that there are at least $N$ distinct such types.

The fact that there are only $N$ possible types of pairs $(r_i, r_j)$ follows from the fact that for any $i, k \in \{1, \ldots, N\}$, $\tp(r_i / \overline{f}) = \tp(r_k / \overline{f})$.
\end{proof}

\begin{definition}
\label{G_notation}
Fix some $f \in P$ and let $g_i = \pi^n_i(f)$ for each $i \in \{1, \ldots, n\}$. Suppose that $N$ is the number of realizations of the type $q := \tp(f / g_1, \ldots, g_n)$, and let $f = f_1, f_2, \ldots, f_N$ be an enumeration of all the realizations of $q$. 

By Lemma~\ref{pair_types}, there are formulas $\varphi(\overline{z}) = \varphi(z_1, \ldots, z_n)$, $\psi(y,\overline{z})$, and $\theta_1(y_1, y_2, \overline{z}), \ldots, \theta_N(y_1, y_2, z)$ such that:
\begin{enumerate}
\item $\varphi(\overline{z}) \in \tp(g_1, \ldots g_n)$;
\item $\psi(y, \overline{z}) \in \tp(f, g_1, \ldots, g_n)$ and $\psi(y, \overline{z}) \vdash \varphi(\overline{z})$;
\item $\theta_i(y_1, y_2, \overline{z}) \in \tp(f_1, f_i, g_1, \ldots, g_n)$ and $\theta_i(y_1, y_2, \overline{z}) \vdash \psi(y_1, \overline{z}) \wedge \psi(y_2, \overline{z})$;

\item Whenever $\varphi(\overline{c})$ holds, then $\psi(y, \overline{c})$ has precisely $N$ realizations;
\item Whenever $\psi(f, \overline{c}) \wedge \psi(f', \overline{c})$ holds, there is exactly one $i \in \{1, \ldots, N\}$ such that $\theta_i(f, f', \overline{c})$ holds; and
\item Whenever $\varphi(\overline{c})$ and $\psi(f, \overline{c})$ hold, then $\overline{c} = \pi^n(f) = (\pi^n_1(f), \ldots, \pi^n_n(f))$.
\end{enumerate}
\end{definition}

\begin{definition}
\label{group_comp}
Suppose that $f \models p_n$ and $(g_1, \ldots, g_n) = \pi^n(f) $ (so $\models \varphi(\overline{g})$).
Then for any two $i, j \in \{1, \ldots, N\}$, there is a unique element $i \star j \in \{1, \ldots, N\}$ such that whenever $\theta_i(f_2, f_3) \wedge \theta_j(f_1, f_2)$ holds, then $\theta_{i \star j}(f_1, f_3)$ also holds. This forms a group operation on the set $\{1, \ldots, N\}$, and we call the resultant group $G$.

Furthermore, from now on we will assume that the formulas in Definition~\ref{G_notation} satisfy the following additional property:

\begin{enumerate}
\setcounter{enumi}{6}
\item Whenever $\theta_i(f_2, f_3, \overline{c}) \wedge \theta_j(f_1, f_2, \overline{c})$ holds, then $\theta_{i \star j}(f_1, f_3, \overline{c})$ also holds.
\end{enumerate}
\end{definition}

\begin{definition}
\label{G_f}
Suppose that $\varphi(\overline{c})$ holds.
\begin{enumerate}
\item The equivalence relation $\sim_{\overline{c}}$ on $\psi(y, \overline{c}) \times \psi(y, \overline{c})$ is defined by: $(r,s) \sim_{\overline{c}} (r',s')$ if and only if there is some $i \in \{1, \ldots, N\}$ such that $\theta_i(r,s, \overline{c}) \wedge \theta_i(r', s', \overline{c})$. We use brackets $[ (r,s) ]$ to denote the $\overline{c}$-class of $(r,s)$. (Note that it is never ambiguous what $\overline{c}$ is since $\overline{c} \in \dcl(r)$.)
\item $G_{\overline{c}}$, as a set, is the quotient of $\psi(y, \overline{c}) \times \psi(y, \overline{c})$ by the equivalence relation $\sim_{\overline{c}}$. We define a group operation $\cdot$ on $G_{\overline{c}}$ by the rule that $$[(s,t)] \cdot [(r,s)] = [(r,t)].$$ (By the properties of the formulas $\theta_i$, this gives a well-defined binary operation.)
 \end{enumerate}
\end{definition}

\begin{lemma}
For any two $n$-tuples $\overline{c}$ and $\overline{c}'$ such that $\varphi(\overline{c}) \wedge \varphi(\overline{c}')$ holds, there is a canonical $(\overline{c},\overline{c}')$-definable isomorphism $\Phi_{\overline{c},\overline{c}'} : G_{\overline{c}} \rightarrow G_{\overline{c}'}$. These automorphisms commute: $\Phi_{\overline{c}', \overline{c}''} \circ \Phi_{\overline{c}, \overline{c}'} = \Phi_{\overline{c}, \overline{c}''}$.
\end{lemma}

\begin{proof}
Let $\Phi_{\overline{c},\overline{c}'}([(r,s)]) = [(r', s')]$ if and only if there is some $i \in \{1, \ldots, N\}$ such that $$\models \theta_i(r, s, \overline{c}) \wedge \theta_i(r', s', \overline{c}').$$

\end{proof}

At this point it is clear that the isomorphism type of all the groups $G_{\overline{c}}$ in the previous Lemma is the same as the $G$ in Definition~\ref{group_comp} above. Using the definability of $\sim_{\overline{c}}$, we will think of $G$ as living in $\mathfrak{C}^{eq} = \mathfrak{C}$ and identify it with the concrete groups $G_{\overline{c}}$. Although it is not immediate from the definition, we will see below that the group $G$ is always abelian.

\begin{definition}
\label{std_action_G}
Given a tuple $\overline{c}$ such that $\varphi(\overline{c})$ holds, the \emph{standard action of $G$ on $\psi(\mathfrak{C}, \overline{c})$} is given by the rule: if $\alpha = [(r,s)]$ where $$\models \psi(r, \overline{c}) \wedge \psi(s,\overline{c}) \wedge \theta_i(r, s, \overline{c}),$$ then for any $f \in \psi(\mathfrak{C}, \overline{c})$, $\alpha(f)$ is the unique element such that $\theta_i(f, \alpha(f), \overline{c})$ holds.

\end{definition}

\begin{lemma}
\label{alt}
Suppose that $Q(r_1, \ldots, r_{n+1})$ holds of some compatible $(n+1)$-tuple $(r_1, \ldots, r_{n+1})$ from $P_n$. Then for any $\alpha \in G$ and any $i \in \{1, \ldots, n\}$, $$\models Q(r_1, \ldots, r_{i-1}, \alpha(r_i), \alpha(r_{i+1}), r_{i+2}, \ldots, r_{n+1}).$$
\end{lemma}

\begin{proof}
Suppose that $(a_1, \ldots, a_{n+1})$ is the realization of $I^{(n+1)}$ such that for any $k \in \{1, \ldots, {n+1}\}$, $\pi(r_k) = (a_1, \ldots, \widehat{a}_k, \ldots, a_{n+1})$, and pick any $a_{n+2}$ realizing the nonforking extension of $\tp(a_1)$ to $(a_1, \ldots, a_{n+1})$. Throughout the argument below, we let $s_k := [n+2] \setminus \{k\}$ and $s_{k,\ell} := [n+2] \setminus \{\ell, k\}$, and we use the same symbol $s_k$ or $s_{k,\ell}$ to denote the ordered tuple enumerating the elements of this set in ascending order.

To ease notation, let $$\overline{A}_k = \ov{A}_{s_k} =  \acl(a_1, \ldots, \widehat{a_k}, \ldots, a_{n+2})$$ and $$\overline{A}_{k, \ell} = \ov{A}_{s_{k,\ell}} = \acl(a_1, \ldots, \widehat{a_k}, \ldots, \widehat{a_\ell}, \ldots, a_{n+w}).$$ 

Use Lemma~\ref{n_system_existence} to fix some $(n-1)$-symmetric system $\{\ov{a}_s : s \in [n+2]^{(\leq n-1)} \}$ for $\{a_1, \ldots, a_{n+2}\}$, and let ``$\ov{a}_{\bd(s_{i, n+2})}$'' be the tuple $$(\ov{a}_{s_{1,i, n+2}}, \ldots, \ov{a}_{s_{i, n+1, n+2}}) .$$ 

Let $r'_{i+1}$ be the unique element such that $$\models Q(r_1, \ldots, r_{i-1}, \alpha(r_i), r'_{i+1}, r_{i+2}, \ldots, r_{n+1}).$$ We will show that $$(*) \hspace{.3in} (r_i, \alpha(r_i), \overline{a}_{\bd(s_{i,n+2})}) \equiv (r_{i+1}, r'_{i+1}, \overline{a}_{\bd(s_{i+1, n+2})}),$$ which immediately implies that $r'_{i+1} = \alpha(r_{i+1})$, finishing the proof of the Lemma.


Pick some compatible $n$-tuple $\overline{g} = \langle g_k : k \in [n+2] \setminus \{i, i+1\} \rangle$ from $P_{n-1}$ such that $$\pi(g_k) = (a_1, \ldots, \widehat{a}_i, \widehat{a}_{i+1}, \ldots, \widehat{a}_k, \ldots, a_{n+2})$$ and $g_{n+2} = \pi^n_i(r_{i+1}) = \pi^n_{i+1}(r_i)$. To prove $(*)$, it suffices to show:

$$(**) \hspace{.3in} \chi(1, i+1; r_i, \alpha(r_i), \overline{g}) = \chi(1, i+1; r_{i+1}, r'_{i+1}, \overline{g}).$$

\begin{claim}
\label{sigma_0}
There is a system of elementary maps $\langle \varphi_s : s \subseteq s_0 \rangle$ satisfying all of the following properties:
\begin{enumerate}
\item $\varphi_s \in \Aut(\overline{A}_s)$;
\item If $s \subseteq t$, then $\varphi_s \subseteq \varphi_t$;
\item If $\{1, \ldots, \widehat{i}, i+1, \ldots, n+1\} \subseteq s$, then $\varphi_s(r_i) = \alpha(r_i)$;
\item If $\{1, \ldots, i, \widehat{i+1}, \ldots, n+1\} \subseteq s$, then $\varphi_s(r_{i+1}) = r'_{i+1}$;
\item If $|s| < n$, then $\varphi_s$ is the identity map; and
\item If $k \in \{1, \ldots, n+1\} \setminus \{i, i+1\}$ and $s \subseteq [n+2] \setminus \{k\}$, then $\varphi_s$ is the identity map.

\end{enumerate}
\end{claim}

To prove Claim~\ref{sigma_0}, first pick the maps $\varphi_s$ for every $s \subseteq s_{i+1, n+2}$ satisfying conditions~(4) and (5) using the fact that (by our assumptions) $$(r_{i+1}, \overline{a}_{\bd(s_{n+2})}) \equiv (r'_{i+1}, \overline{a}_{\bd(s_{n+2})}).$$ Then use relative $(n,n)$-uniqueness over the base set $\acl(a_i)$ to find a map $\varphi_{s_{n+2}} \in \Aut(\acl(a_1, \ldots, a_{n+1}))$ which extends $\varphi_{s_{ i+1, n+2}}$ and the identity maps on $\ov{A}_{s_{k, n+2}}$ for every $k \in [n+1] \setminus \{i,i+1\}$. The fact that condition~(3) is satisfied is immediate from the definable $Q$ relation which is preserved by the $\sigma$ maps.

\begin{claim}
\label{sigma_full}

We may extend the system of maps in Claim~\ref{sigma_0} to a system of maps $\langle \varphi_s : s \subseteq [n+2]\rangle$ which continues to satisfy (1)-(6).

\end{claim}

To prove Claim~\ref{sigma_full}, we simply use relative $(n,n)$-uniqueness over the base set $\acl(a_i, a_{i+1})$ to amalgamate the map $\varphi_{s_{n+2}}$ from Claim~\ref{sigma_0} with the identity maps on $\ov{A}_{s_k}$ for each $k \in [n+1] \setminus \{i, i+1\}$ (noting that this is coherent by condition~(5) of Claim~\ref{sigma_0}).

Now that we have Claim~\ref{sigma_full}, write $\varphi$ for the map $\varphi_{[n+2]} \in \Aut(\ov{A}_{[n+2]})$. Fix some arbitrary $k \in P_n(\overline{g})$ and elements $\langle h_k, h'_k : k \in \{1, \ldots, n+1\} \setminus \{i, i+1\} \rangle$ such that $$\pi(h_k) = (a_1, \ldots, \widehat{a}_i, \ldots, \widehat{a}_k, \ldots, a_{n+2}),$$ $$\pi(h'_k) = (a_1, \ldots, \widehat{a}_{i+1}, \ldots, \widehat{a}_k, \ldots, a_{n+2}),$$ and $$(***) \hspace{.2in} Q(r_i, h_1, \ldots, h_{i-1}, k, h_{i+2}, \ldots, h_{n+1}) \wedge Q(r_{i+1}, h'_1, \ldots, h'_{i-1}, k, h'_{i+2}, \ldots, h'_{n+1}).$$ Then applying the map $\varphi$ to both relations in $(***)$, we have that $$Q(\alpha(r_i), h_1, \ldots, h_{i-1}, \varphi(k), h_{i+2}, \ldots, h_{n+1}) \wedge Q(r'_{i+1}, h'_1, \ldots, h'_{i-1}, \sigma(k), h'_{i+2}, \ldots, h'_{n+1}).$$ Comparing the two $Q$ relations above, we have that, on the one hand, $$\varphi(k) = \left[\chi(1, i+1; r_i, \alpha(r_i), \overline{g})\right](k),$$ and on the other hand, $$\varphi(k) = \left[\chi(1, i+1; r_{i+1}, r'_{i+1}, \overline{g})\right](k).$$ Since $k \in P_n(\overline{g})$ was arbitrary, we conclude that $\chi(1, i+1; r_i, \alpha(r_i), \overline{g}) = \chi(1, i+1;r_{i+1}, r'_{i+1}, \overline{g})$, establishing $(**)$ and hence the Lemma.

\end{proof}

\begin{corollary}
\label{G_abelian}
$G$ is abelian.
\end{corollary}

\begin{proof}
This follows by the same argument as in Proposition~\ref{alt_2}~(2), using Lemma~\ref{alt} plus the regularity of the action of $G$ on each set $\psi(\C, \overline{c})$ such that $\models \varphi(\overline{c})$.
\end{proof}

\begin{lemma}
\label{strong_skel_uniq}
Suppose that $k \geq n$, $(a_1, \ldots, a_k), (b_1, \ldots, b_k) \in I^{(k)}$, and that $(A_1, \ldots, A_{n-1})$ and $(B_1, \ldots, B_{n-1})$ are maximal compatible systems going up to $P_{n-1}$ in the sense of Definition~\ref{compatible} such that $A_1 = \{a_1, \ldots, a_k\}$ and $B_1 = \{b_1, \ldots, b_k\}$. Let $$\acl_i(\overline{a}) = \bigcup_{u \subset_i \{1, \ldots, k\}} \acl(\overline{a}_u),$$ and $\acl_i(\overline{b})$ is defined similarly. Then there is an elementary map $f : \acl_{n-1}(\overline{a}) \rightarrow \acl_{n-1}(\overline{b})$ such that $f(a_i) = b_i$ for every $i \in \{1, \ldots, k\}$ and $f(A_j) = B_j$ for every $j \in \{1, \ldots, n-1\}$.
\end{lemma}

\begin{proof}
Apply Lemma~\ref{skeletal_maps} and the ``isolation of types'' clause for a symmetric witness to the failure of $(n+1)$-uniqueness (part~(1) of Definition~\ref{symm_wit}).



\end{proof}

Finally, we put all the pieces together to prove the main theorem of the section.

\emph{Proof of Theorem~\ref{polygroup_definability}:}
Recall that we are assuming $T$ is stable with $k$-uniqueness for every $k \leq n$ and that $T$ does not have $(n+1)$-uniqueness over $A$. By Proposition~\ref{symm_wit_existence}, there is a symmetric witness to the failure of $(n+1)$-uniqueness over $A$. In other words, there is an infinite Morley sequence $I$ over some $A = \acl(A)$ and a relatively $I$-definable connected $n$-ary quasigroupoid $\mathcal{H} = (I, P_2, \ldots, P_n, Q)$ satisfying (1) and (2) of Definition~\ref{symm_wit}. Clause (2) of this definition implies that $\mathcal{H}$ is locally finite.

The group $G$ defined in Definition~\ref{group_comp} and its standard action (Definition~\ref{std_action_G}) form a regular action on the quasigroupoid $\mathcal{H}$ by Lemma~\ref{alt}. By Lemma~\ref{strong_skel_uniq} above, this action is $(n+2)$-homogeneous (as in Definition~\ref{homog_action}).  Now Theorem~\ref{associativity_action} above implies that, possibly after replacing $Q$ by $Q_g$ for some $g \in G$ (which certainly preserves relative $I$-definability), we may assume that $Q$ is associative. So $\mathcal{H}$ is an $n$-ary polygroupoid. Finally, condition~(4) of Theorem~\ref{polygroup_definability} is satisfied by $\mathcal{H}$ due to clause~(1) of Definition~\ref{symm_wit} (``isolation of types'').


For the converse, suppose that we have a Morley sequence $I$ over $A$ and a relatively $I$-definable $n$-ary polygroupoid $\mathcal{H}$ satisfying (1)-(4) of the hypothesis. Then for any $(a_1, \ldots, a_{n+1}) \in I^{(n+1)}$ and any $f \in P_n(a_1, \ldots, a_n)$, evidently $f \in \acl_A(a_1, \ldots, a_n)$ (because $\mathcal{H}$ is locally finite), $f \notin \bd_A(a_1, \ldots, a_n)$ (by clause~(4) and the fact that $G \neq 1$), and $$f \in \dcl(\acl_A(\widehat{a}_1 a_2 \ldots a_{n+1}), \acl_A(a_1 \widehat{a}_2 \ldots a_{n+1}), \ldots, \acl_A(a_1 \ldots \widehat{a}_n a_{n+1})$$ (by picking appropriate elements $f_i \in P_n(a_1, \ldots, \widehat{a}_i, \ldots, a_{n+1})$ and using the $Q$ relation). Therefore $T$ fails to have the property defined as $B(n+1)$ above which is equivalent to $(n+1)$-uniqueness in any stable theory with $n$-uniqueness (Fact~\ref{Bn_uniqueness}). $\dashv$

\bigskip

It is worth noting that the $n$-ary polygroupoid constructed above is totally categorical:

\begin{corollary}
If $T$ is a stable theory with $(\leq n)$ uniqueness but not $(n+1)$-uniqueness, then there is some Morley sequence $I$ and some relatively $I$-definable $n$-ary polgroupoid $\mathcal{H} = (I, P_2, \ldots, P_n, Q)$ witnessing the failure of $(n+1)$-uniqueness such that the complete theory of $\mathcal{H}$ is totally categorical, in the $n$-sorted language with sorts for $I, P_2, \ldots, P_n$, a relation symbol for $Q$, and function symbols for the projection maps $\pi^i : P_i \rightarrow (P_{i-1})^i$.

More precisely, there is a relatively $I$-definable, totally categorical $n$-ary polygroupoid as above such that some closed independent $(n+1)$-amalgamation problem $\mathcal{A}$ in $\mathcal{H}$ has two non-isomorphic solutions, and this yields a closed independent $(n+1)$-amalgamation problem in the original theory $T$ with two non-isomorphic solutions (since the set $I$ is independent and the definable structure on $\mathcal{H}$ is definable from $T$).

\end{corollary}

\begin{proof}
Let $\mathcal{H}$ be the $n$-ary polygroupoid constructed in Theorem~\ref{polygroup_definability}. At this point it is clear that there is a closed independent $(n+1)$-amalgamation problem over $\acl(\emptyset)$ in $\mathcal{H}$ with two non-isomorphic solutions (use the algebraic closures of independent elements of $I$ as the ``vertices'' $\mathcal{A}(\{i\})$, and the non-uniqueness of solutions comes from the interdefinability of fibers in $P_n$ given by the $Q$ relation).

For the $\omega$-categoricity of $\mathcal{H}$, we show that for every $k > 0$ there are finitely many complete $k$-types over $\emptyset$.  Since $I$ is an indiscernible set in $T$, for every permutation $\sigma$ of $I$ there is an automorphism $\varphi_\sigma$ of the structure $\mathcal{H}$ which extends $\sigma$. Therefore if $A$ is any set of $n \cdot k$ distinct elements of $I$, any orbit of any $(f_1, \ldots, f_k)$ under $\Aut(\mathcal{H})$ intersects $\cl(A)$. But since $\mathcal{H}$ is locally finite, $\cl(A)$ is finite.

Finally, to show that the theory of $\mathcal{H}$ is uncountably categorical, note that for any finite $A \subseteq \mathcal{H}$, the set $I \setminus \supp(A)$ is an indiscernible set over $A$ (since any permutation of $I \setminus \supp(A)$ can be extended to an elementary map in $\Aut(\mathfrak{C} / A)$, and hence to an automorphism of $\mathcal{H}$ fixing $A$). Thus $I$ is a strongly minimal set in the theory of $\mathcal{H}$, and since $\mathcal{H}$ is the algebraic closure of $I$, it is uncountably categorical.

\end{proof}

\section{Examples of $n$-ary quasigroupoids and $n$-ary polygroupoids}

In this section, we describe a family of examples of connected $n$-ary quasigroupoids and $n$-ary polygroupoids. These are the simplest possible $n$-ary quasigroupoids equipped with a regular action of a finite abelian group $G$: in these examples, the fibers $P_2, \ldots, P_{n-1}$ are all ``trivial'' in the sense that for any $k \in \{2, \ldots, n-1\}$ and any $(a_1, \ldots, a_k) \in I^{(k)}$, there is a unique element in $P_k(a_1, \ldots, a_k)$. We will show that their complete theories (in a suitable language) are totally categorical, eliminate quantifiers, have $k$-uniqueness for every $k \leq n$, and fail to have $(n+1)$-uniqueness.

The language $L_n$ for these examples is as follows. There are $(n+1)$ sorts $I = P_1, P_2, \ldots, P_n = P$ and $G$, plus function symbols for $\pi^k : P_k \rightarrow (P_{k-1})^k$. There are function symbols for a binary operation on $G$ and for an action from $(G \times P)$ to $P$, both of which we will denote by ``$+$'' without risk of confusion. There is an $(n+1)$-ary relation symbol $Q$ on $P$. Finally, for every $\sigma$ in the symmetric group $S_n$, there is a function symbol $\iota_\sigma : P \rightarrow P$.

Next we describe the standard countable $L_n$-structure $\mathcal{H}_{G,n}$ whose complete theory will be $T_{G,n}$. This structure depends on choices of both an integer $n \geq 2$ and a finite nontrivial abelian group $G$.

In $\mathcal{H}_{G,n}$, the sort $I$ is interpreted as $\omega$ and for every $k \in \{2, \ldots, n-1\}$, the sort $P_k$ is simply $\omega^{(k)}$. The sort $P_n$ is interpreted as $\omega^{(n)} \times G$. The projection functions $\pi^k : P_k \rightarrow (P_{k-1})^k$ are interpreted in the natural way: for $k < n$ and $i \in \{1, \ldots, k\}$, $$\pi^k_i ((a_1, \ldots, a_k)) = (a_1, \ldots, \widehat{a}_i, \ldots, a_k),$$ and $\pi^n$ is given by $$\pi^n_i((a_1, \ldots, a_n), g) = (a_1, \ldots, \widehat{a}_i, \ldots, a_n).$$ If $(f_1, \ldots, f_{n+1})$ is a compatible $(n+1)$-tuple from $P$ such that $f_i = ((a_1, \ldots, \widehat{a}_i, \ldots, a_{n+1}), g_i)$, then we define $Q$ so that $$Q(f_1, \ldots, f_{n+1}) \Leftrightarrow \sum_{i=1}^{n+1} (-1)^i g_i = 0.$$ The group $G$ acts on $P$ by the rule $$((a_1, \ldots, a_n), g) + h = ((a_1, \ldots, a_n), g+h).$$ Finally, the function symbols $\iota_\sigma : P \rightarrow P$ are interpreted by the rule $$\iota_\sigma( (a_1, \ldots, a_n), g) = (a_{\sigma(1)}, \ldots, a_{\sigma(n)}), \textup{sign}(\sigma)).$$

The following facts about $\mathcal{H}_{G,n}$ are all easily checked from the definition of the structure:

\begin{proposition}
\label{std_model_Tng}
\begin{enumerate}
\item $\mathcal{H}_{G,n}$ is a connected $n$-ary quasigroupoid.
\item For any $k \in \{2, \ldots, n-1\}$ and any $(a_1, \ldots, a_k) \in I^{(k)}$, $|P_k(a_1, \ldots, a_k)| = 1$.
\item $\mathcal{H}_{G,n}$ is associative (hence an $n$-ary polygroupoid).
\item The action of $G$ on $\mathcal{H}_{G,n}$ is regular.
\item For any $\sigma \in S_n$ and $(a_1, \ldots, a_n) \in I^{(n)}$, $\iota_\sigma$ induces a bijection from $P(a_1, \ldots, a_n)$ to $P(a_{\sigma(1)}, \ldots, a_{\sigma(n)})$.
\item For any $\sigma, \tau \in S_n$, $$\iota_{\sigma \circ \tau} = \iota_\sigma \circ \iota_\tau.$$
\item For any $g \in G$, any $\sigma \in S_n$, and any $f \in P$, $$\iota_\sigma (f + g) = \iota_\sigma(f) + (-1)^{\textup{sign}(g)}(g).$$
\item If $f_1, \ldots, f_{n+1}$ is a compatible tuple from $P$ such that $Q(f_1, \ldots, f_{n+1})$ holds, then for any $i \in \{1, \ldots, n-1\}$, if $\sigma_i \in S_n$ is the transposition $(i, i+1)$, then $$Q(\iota_{\sigma_{i-1}}(f_1), \ldots, \iota_{\sigma_{i-1}}(f_{i-1}), f_{i+1}, f_i, \iota_{\sigma_i}(f_{i+2}), \ldots, \iota_{\sigma_i}(f_{n+1}) ).$$
\end{enumerate}
\end{proposition}

We think of the maps $\iota_\sigma$ as ``generalized inverse maps'' on the sort $P$, noting that when $n=2$ and $\sigma = (1 2)$ this really is an inverse map from $\mor(a_1,a_2)$ to $\mor(a_2, a_1)$. It turns out that is necessary to add the maps $\iota_\sigma$ to the basic language if we hope to have elimination of quantifiers. For instance, in the case where $n=3$, consider any $(a_1, a_2, a_3, a_4) \in I^{(4)}$ and pair of elements $(f_1, f_2)$ from $P$ such that $f_1 \in P(a_1, a_2, a_4 )$ and $f_2 \in P(a_2, a_1, a_{4})$. From the pair $(f_1, f_2)$, we can define a bijection $\iota_{(f_1, f_2)}$ from $P(a_1, a_2, a_3)$ to $P(a_2, a_1, a_3)$ by the rule that whenever $$Q(h_1, h_2, f_1, f)$$ holds, then so does $$Q(h_2, h_1, f_2, \iota_{(f_1,f_2)}(f)).$$ As is easy to verify, this definition is independent of the choices of $h_1$ and $h_2$. Thus over $\acl^{eq}(a_1, a_2)$ we have a family of definable bijections from $P(a_1, a_2, b)$ to $P(a_2, a_1, b)$, which it turns out is equal to $\iota_{(12)} + g$ for some fixed $g \in G$.

\begin{definition}
\label{polygroup_with_inverses}
A \emph{connected $n$-ary polygroupoid with inverses} is an $L_n$-structure satisfying all the properties (1) through (8) of Proposition~\ref{std_model_Tng}.
\end{definition}

\begin{definition}
\label{star_set}
If $(I, <)$ is any linearly ordered set such that $| I | \geq n$ with a minimal element $a \in I$, then the \emph{$n$-star of $(I,<)$} is the set $$\Star_n(I, <) := \{ (a, b_1, \ldots, b_{n-1}) : a < b_1 < \ldots < b_{n-1} \}.$$

A set $\mathcal{S} \subseteq I^{(n)}$ is an \emph{$n$-star centered at $a$} or an \emph{$n$-star over $I$} if $\mathcal{S} = \Star_n(I, <)$ for some linear ordering $<$ on $I$ in which $a$ is the minimal element.

If $\mathcal{S}$ is an $n$-star, a \emph{solution on $\mathcal{S}$} is a function  $s : \mathcal{S} \rightarrow P_n$ such that for any $u \in \mathcal{S}$ we have that $s(u) \in P_n(u)$.
\end{definition}

\begin{proposition}
\label{isomorphisms_stars}
Suppose $\mathcal{H} = (I, P_2, \ldots, P_n, Q, G) $ and $\mathcal{H}^* = (I^*, P^*_2, \ldots P^*_n, Q^*, G^*)$ are two connected $n$-ary polygropuoids with inverses with $|I| \geq n$, and that we have the following:

\begin{enumerate}
\item A bijection $\varphi: I \rightarrow I^*$;
\item A group isomorphism $\psi: G \rightarrow G^*$;
\item An $n$-star $\mathcal{S}$ over $I$ centered at $a$, which corresponds via $\varphi$ to an $n$-star $\mathcal{S}^* := \varphi(\mathcal{S})$ centered at $a^* := \varphi(a)$; and
\item Two solutions $s : \mathcal{S} \rightarrow P_n$ and $s^* : \mathcal{S}^* \rightarrow P^*_n$.
\end{enumerate}

Then there is a unique isomorphism $\chi : \mathcal{H} \rightarrow \mathcal{H}^*$ of $L_n$-structures which extends $\varphi \cup \psi$ and satisfies

$$\chi(s(u)) = s^*(\varphi(u))$$

for every $u \in \mathcal{S}$.

\end{proposition}

\begin{proof}
First, we fix a linear ordering $<$ over $I$ such that $\mathcal{S} = \Star_n(I, <)$ and $a$ is the $<$-minimal element of $I$ (in fact, for $n \geq 3$, this ordering is uniquely determined by the set $\mathcal{S}$). This transfers via $\varphi$ to a linear ordering over $I^*$ with $a^*$ as the minimal element such that $\mathcal{S}^* = \Star_n(I^*, <)$, and we will use the same symbol $<$ for both orderings without risk of confusion.

Form the partition $I^{(n)} = I_a \cup I_b$ where $I_a$ consists of all tuples $I^{(n)}$ which contain the element $a$ and $I_b := \left( I \setminus \{a\} \right)^{(n)}$. Let $P = P_a \cup P_b$ be the corresponding partition of $P = P_n$, and we define sets $I^*_a, I^*_b, P^*_a,$ and $P^*_b$ in $\mathcal{H}_2$ analogously (replacing $a$ everywhere by $a^*$).

The first step is to define a function $\varphi_a : P_a \rightarrow P^*_a$ using the inverse maps. Given any tuple $(c_1, \ldots, c_n) \in I_a$, first select the unique $(b_1, \ldots, b_n) \in \mathcal{S}$ and the unique $\sigma \in S_n$ such that $$(c_1, \ldots, c_n) = (b_{\sigma(1)}, \ldots, b_{\sigma(n)}).$$ Then for any $f \in P(c_1, \ldots, c_n)$, there is a unique $g \in G$ such that $$f = \iota_\sigma \left[ s(b_1, \ldots, b_n) \right] + g,$$ and we define $$\varphi_a(f) := \iota^*_\sigma \left[ s^*( \varphi(b_1), \ldots, \varphi(b_n)) \right] + \psi(g).$$

Observe that it follows immediately from the definition that $\varphi_a$ respects the action of the group $G$, that is, that $\varphi_a(f + g) = \varphi_a(f) + \psi(g)$ for any $f \in \dom(\varphi_a)$ and any $g \in G$. It is also immediate that $\varphi_a$ satisfies the requirement that $\varphi_a(s(u)) = s^*(\varphi(u))$ for any $u \in \mathcal{S}$, simply by letting $\sigma$ be the identity permutation and by letting $g = 0$.

Now we define a function $\varphi_b : P_b \rightarrow P^*_b$ via $\varphi^a$ and the $Q$-relation: if $f \in P_b$, say $f \in P(c_1, \ldots, c_n)$, then we can pick elements $f_i \in P(a, c_1, \ldots, \widehat{c_i}, \ldots, c_n)$ such that $Q(f, f_1, \ldots, f_n)$ holds, and then define $\varphi_b(f)$ to be the unique element of $P^*(\varphi(c_1), \ldots, \varphi(c_n))$ such that $$Q^*(\varphi_b(f), \varphi_a(f_1), \ldots, \varphi_a(f_n)).$$

The first thing to check is that $\varphi_b$ is well-defined, that is, that the definition of $\varphi_b(f)$ does not depend on the choice of the elements $f_i$. If we also have $Q(f, h_1, \ldots, h_n)$ where $h_i \in P(a, c_1, \ldots, \widehat{c_i}, \ldots, c_n)$, then $h_i = f_i + g_i$ for some $g_i \in G$, and by Proposition~\ref{alt_2}, $$\sum_{i = 1}^n (-1)^i g_i = 0.$$ The alternating sum of the elements $\psi(g_i)$ is also $0$ since $\psi$ is a group isomorphism. Therefore by Proposition~\ref{alt_2} again, $$Q^*(\varphi_b(f), \varphi_a(f_1) + \psi(g_1), \ldots, \varphi_a(f_n) + \psi(g_n)),$$ and by the observation above that $\varphi_a$ commutes with the action of the group $G$, we conclude that $Q^*(\varphi_b(f), \varphi_a(h_1), \ldots, \varphi_a(h_n))$.

We now define $\chi$ as the function from $\mathcal{H}$ to $\mathcal{H}^*$ which acts as $\varphi$ on the $I$ sort, as $\psi$ on the $G$ sort, as $\varphi^a \cup \varphi^b$ on the $P_n$ sort, and commutes with the projection functions $\pi^i$. Evidently this $\chi$ induces bijections between all the sorts which respect the fibers in $P$. We must check that $\chi$ respects the rest of the $L_n$-definable structure: that it commutes with the action of $G$ and the inverse maps $\iota_\sigma$, and that it preserves the $Q$ relation.

\textbf{$\chi$ respects the action of $G$:} Let $f \in P(u)$ and $g \in G$. As already observed above, it follows directly from the definition of $\varphi_a$ that if $u \in P_a$, then $\chi(f + g) = \chi(f) + \chi(g)$. So suppose that $g \in G$ and $f \in P_b$ and $$Q(f, f_1, \ldots, f_n)$$ for elements $f_i \in P_a$. We must check that $\varphi_b(f + g) = \varphi_b(f) + \psi(g)$. But by the definition of $\varphi_b$ and Proposition~\ref{alt_2}, $$Q^*(\varphi_b(f) + \psi(g), \varphi_a(f_1) + \psi(g), \varphi_a(f_2), \ldots, \varphi_a(f_n))$$ holds, and since $\varphi_a$ commutes with the $G$ action, $$Q^*(\varphi_b(f) + \psi(g), \varphi_a(f_1 + g), \varphi_a(f_2), \ldots, \varphi_a(f_n)).$$ By Proposition~\ref{alt_2} again, we also have that $$Q(f + g, f_1 + g, f_2, \ldots, f_n),$$ and so $$Q^*(\varphi_b(f + g), \varphi_a(f_1 + g), \varphi_a(f_2), \ldots, \varphi_a(f_n)).$$ Comparing the last two sentences, we conclude that $\varphi_b(f+ g) = \varphi_b(f) + \psi(g)$.

\textbf{$\chi$ commutes with the inverse maps:} Fix $\sigma \in S_n$ and $f \in P$, and we check that $\varphi(\iota_\sigma(f)) = \iota^*_\sigma(\varphi(f))$. We break into cases depending on whether $f \in P_a$ or $f \in P_b$.

First suppose that $f \in P_a$, and say $f \in (c_1, \ldots, c_n)$ where some $c_i$ is equal to $a_1$. As in the definition of $\varphi_a$ above, we first fix $(b_1, \ldots, b_n) \in \mathcal{S}$ and $\tau \in S_n$ such that $$(c_1, \ldots, c_n) = (b_{\tau(1)}, \ldots, b_{\tau(n)})$$ and we pick $g \in G$ such that $$f = \iota_\tau ( s(b_1, \ldots, b_n) ) + g.$$ We calculate: $$\chi \left( \iota_\sigma(f) \right) = \varphi_a \left( \iota_{\sigma \circ \tau} (s(b_1, \ldots, b_n)) + (-1)^{\textbf{sg}(\sigma)} g \right),$$ which by the definition of $\varphi_a$ is equal to $$\iota^*_{\sigma \circ \tau} \left[s^*(\varphi(b_1), \ldots, \varphi(b_n)) \right] +  (-1)^{\textbf{sg}(\sigma)} \psi(g) $$ $$= \iota^*_{\sigma} \left[\iota^*_\tau (s^*(\varphi(b_1), \ldots, \varphi(b_n))) + \psi(g) \right]$$ $$= \iota^*_\sigma (\varphi_a(f)), $$ which is what we wanted to show.

For the other case, suppose $f \in P_b$ and pick elements $f_j \in P_a$ (for $j \in \{1, \ldots, n\}$) such that $$Q(f, f_1, \ldots, f_n),$$ and so by definition of $\varphi_b$, $$Q^*(\varphi_b(f), \varphi_a(f_1), \ldots, \varphi_a(f_n)).$$ By property~(8) of Proposition~\ref{std_model_Tng} (the preservation of $Q$ under transpositions) applied to both of the $Q$ relations immediately above, we conclude that $$(\dag) \hspace{.1in} Q(\iota_{\sigma_i}(f), \iota_{\sigma_i}(f_1), \ldots, f_{i+1}, f_i, \iota_{\sigma_{i+1}}(f_{i+2}), \ldots, \iota_{\sigma_{i+1}}(f_n))$$ and $$Q^*(\iota^*_{\sigma_i}(\varphi_b(f)), \iota^*_{\sigma_i}(\varphi_a(f_1)), \ldots, \varphi_a(f_{i+1}), \varphi_a(f_i), \iota^*_{\sigma_{i+1}}(\varphi(f_{i+2})), \ldots, \iota^*_{\sigma_{i+1}}(\varphi_a(f_n))).$$ Since we know that $\varphi_a$ commutes with the inverse maps, we derive from the $Q^*$ relation above that $$(\ddag) \hspace{.1in} Q^*(\iota^*_{\sigma_i}(\varphi_b(f)), \varphi_a(\iota_{\sigma_i}(f_1)), \ldots, \varphi_a(f_{i+1}), \varphi_a(f_i), \varphi(\iota_{\sigma_{i+1}}(f_{i+2})), \ldots, \varphi_a(\iota_{\sigma_{i+1}}(f_n))).$$ By $(\dagger)$ and the definition of $\varphi$, we also have that $$(\dag \dag \dag) \hspace{.1in} Q^*(\varphi_b(\iota_{\sigma_i}(f)), \varphi_a(\iota_{\sigma_i}(f_1)), \ldots, \varphi_a(f_{i+1}), \varphi_a(f_i), \varphi_a(\iota_{\sigma_{i+1}}(f_{i+2})), \ldots, \varphi_a(\iota_{\sigma_{i+1}}(f_n))).$$ Comparing $(\ddag)$ with $(\dag \dag \dag)$, we conclude that $\iota^*_{\sigma_i}(\varphi_b(f)) = \varphi_b(\iota_{\sigma_i}(f))$, as desired.

\textbf{$\chi$ respects the $Q$ relation:} Suppose that $h_1, \ldots, h_{n+1}$ are from $P$. It suffices to check that, assuming $Q(h_1, \ldots, h_{n+1})$ holds, then so does $Q^*(\varphi(h_1), \ldots, \varphi(h_{n+1}))$.

Case~(a) is when some $h_i$ belongs to $P_a$. By repeated applications of Property~(8) of Proposition~\ref{std_model_Tng} plus the fact noted above that $\varphi$ commutes with the inverses, it we may assume without loss of generality that $h_1 \in P_a$. But by the very definition of $\varphi_b$ given above, we have that $Q(h_1, h_2, \ldots, h_{n+1})$ implies $Q^*(\varphi_b(h_1), \varphi_a(h_2), \ldots, \varphi_a(h_{n+1}))$, which is what we wanted.

Case~(b) is when \textbf{every} $h_i$ belongs to $P_b$. (This seems to be the only step of the proof which uses the fact that $\mathcal{H}$ and $\mathcal{H}^*$ satisfy the associative law.) First we use the associativity of $\mathcal{H}$ and Remark~\ref{Q_solutions} to pick a collection of elements $\left\{f_{ij} : 1 \leq i < j \leq n+2\right\}$ such that $f_{ij} \in P(a_1, b_2, \ldots, \widehat{b_i}, \ldots, \widehat{b_j}, \ldots, b_{n+2})$, the $Q$ relation holds of every tuple $F_i = (f_{1i}, \ldots, \widehat{f_{ii}}, \ldots, f_{i, n+2})$ (for each $i \in \{1, \ldots, n+2\}$), and $F_1$ is the original tuple $(h_1, \ldots, h_{n+1})$. Then for every $i \in \{2, \ldots, n+2\}$, by case~(a) immediately above, the $Q^*$ relation holds of the tuple $\chi(F_i) := (\chi(f_{1i}), \ldots, \chi(f_{i,n+2}))$. By the associativity of $\mathcal{H}^*$, the $Q^*$ relation also holds of $\chi(F_1)$, that is, $$Q^*(\chi(h_1), \ldots, \chi(h_{n+1})),$$ and we are finished.

Finally, we prove the uniqueness of the isomorphism $\chi: \mathcal{H} \rightarrow \mathcal{H}^*$ satisfying the hypotheses of the Proposition. First note that $\varphi_a$ is the only possible function which maps $P_a$ onto $P^*_a$, commutes with the group isomorphism $\psi: G \rightarrow G^*$ and the inverse maps $\iota_{\sigma}$, and satisfies $\varphi_a(s(u)) = s^*(\varphi(u))$ for every $u \in \mathcal{S}$. Then note that all of $P$ is generated from $P_a$ using the $Q$ relation, so the map $\varphi_a$ has at most one extension to a function on all of $P$ which respects the $Q$ relation. This was all we needed to verify.

\end{proof}

Immediately from Proposition~\ref{isomorphisms_stars}, we obtain the following important corollaries on automorphisms:

\begin{corollary}
\label{automorph_stars}
If $\mathcal{H} = (I, \dots, P, G, Q)$ is any connected $n$-ary polygroupoid with inverses and $S \subseteq I^{(n)}$ is an $n$-star over $I$, then for any function $\varphi : S \rightarrow G$, there is a unique automorphism $\chi \in \Aut(\mathcal{H} / I \cup G)$ such that for every $w \in S$ and every $f \in P(w)$, $$\chi(f) = f + \varphi(w).$$
\end{corollary}

\begin{corollary}
\label{indsicernibility_I}
For any connected $n$-ary polygroupoid with inverses $\mathcal{H} = (I, G, \ldots )$, any permutation $\varphi$ of $I$ can be extended to an automorphism of $\mathcal{H}$ which fixes $G$ pointwise.
\end{corollary}

\begin{lemma}
\label{acl_homesort}
Suppose that $M = (I, P_2, \ldots, P_n, G) \models T_{n,G}$ (with $G$ finite), and let $$M^{home} = I(M) \cup P_2(M) \cup \ldots \cup P_n(M) \cup G(M)$$ (that is, the union of the ``home sorts'' of $M$). Then for any $X \subseteq M^{home}$, $$\acl^{eq}(X) \cap M^{home} = \cl(X)$$ in the sense of Definition~\ref{support}.
\end{lemma}

\begin{proof}
It is immediate that every element of $\cl(X)$ is in $\acl^{eq}(X)$ (this uses the finiteness of $G$). For the other direction, of $\acl^{eq}(X) \cap M^{home}$ contained any element outside of $\cl(X)$, then it would necessarily contain some element $a \in I \setminus (X \cap I)$. But by Corollary~\ref{indsicernibility_I} and the infinitude of $I$, no such element $a$ can belong to $\acl^{eq}(X\cap I) = \acl^{eq}(X)$.
\end{proof}

\begin{corollary}
\label{automorphisms_polygroupoids}

If $\mathcal{H} = (I, P, G, \ldots)$ is any connected $n$-ary polygroupoid with inverses, then $\Gamma = \Aut(\mathcal{H})$ has the following normal series:

$$\Gamma \rhd \Gamma_1 \rhd \Gamma_2,$$

where:

\begin{enumerate}

\item $\Gamma_1 = \Aut(\mathcal{H} / I)$;
\item $\Gamma_2 = \Aut(\mathcal{H} / I \cup G) \cong \prod_{\mathcal{S}} G$, the direct product of infinitely many copies of $G$ indexed by some $n$-star $\mathcal{S}$ over $I$;
\item $\Gamma_2 \lhd \Gamma$;
\item $\Gamma_1 / \Gamma_2 \cong \Aut(G)$; and
\item $\Gamma / \Gamma_2 \cong \textup{Sym}(I)$.
\end{enumerate}
\end{corollary}

\begin{theorem}
\label{total_cat_qe}
If we define $T_{G,n}$ to be the complete $L_n$-theory of the structure $\mathcal{H}_{G,n}$ above (for $n \geq 2$ and $G$ some finite nontrivial abelian group), then:

\begin{enumerate}
\item $T_{G,n}$ is completely axiomatized by a set of sentences asserting the following:
\begin{enumerate}
\item $\mathcal{H}_{G,n}$ is a connected $n$-ary polygroupoid with inverses (that is, it satisfies properties (1) through (8) of Proposition~\ref{std_model_Tng});
\item $I$ is infinite; and
\item A list of axioms for the group structure $(G, +)$ (which characterizes $G$ up to isomorphism, since $G$ is finite).
\end{enumerate}

\item $T_{G,n}$ is totally categorical.
\item $T_{G,n}$ has quantifier elimination after adding constants for each element of $G$.
\end{enumerate}

\end{theorem}

\begin{proof}
We prove all three parts of the theorem simultaneously by a back-and-forth argument using partial $L_n$-isomorphisms determined by $n$-stars. The key point is the following:

\begin{claim}
\label{back_and_forth}
Suppose that $\mathcal{H} = (I, P_2, \ldots, G)$ and $\mathcal{H}^* = (I^*, P^*_2, \ldots, G^*)$ are two connected $n$-ary polygroupoids with inverses with a group isomorphism $\psi: G \rightarrow G^*$. Then given any two sets $A \subseteq I$ and $A^* \subseteq I^*$, any two elements $b \in I \setminus A$ and $b^* \in I^* \setminus A^*$, and any isomorphism of $L_n$-structures $\chi : \cl(A) \rightarrow \cl(A^*)$ extending $\psi$, there is an extension $\tilde{\chi} \supseteq \chi$ to an isomorphism of $L_n$-structures such that $\tilde{\chi}: \cl(A \cup \{b\}) \rightarrow \cl(A^* \cup \{b^*\})$.

\end{claim}

To prove Claim~\ref{back_and_forth}, first pick any $a \in I$ and any $n$-star $\mathcal{S} \subseteq I^{(n)}$ centered at $a$. Let $a^* = \chi(a)$ and let $\mathcal{S}^*$ be the image of $\mathcal{S}$ under $\chi$ (which is an $n$-star centered at $a^*$). Finally, pick an arbitrary function $s: \mathcal{S} \rightarrow P_n$ such that $s(u) \in P_n(u)$ for every $u \in \mathcal{S}$ and let $s^*: \mathcal{S}^* \rightarrow P^*_n$ be the function induced by $\chi$ and $s$, that is, so that $s^*(\chi(u)) = \chi(s(u))$ for every $u \in \mathcal{S}$. Pick any $n$-star $\tilde{\mathcal{S}} \supseteq \mathcal{S}$ (also centered at $a$) and any extension $\tilde{s} \supseteq s$ to a selection function $\tilde{s} : \tilde{\mathcal{S}} \rightarrow P_n$, and let $\tilde{\mathcal{S}}^* \subseteq (I^*)^{(n)}$ and $\tilde{s}^*: \tilde{\mathcal{S}}^* \rightarrow P^*_n$ be induced by $\chi$ as before. Now we define the bijection $\varphi : (A \cup \{b\}) \rightarrow (A^* \cup \{b^*\})$ as the restriction of $\chi$ to $A$ extended so that $\varphi(b) = b^*$, and we apply Proposition~\ref{isomorphisms_stars} using this $\varphi$, the group isomorphism $\psi: G \rightarrow G^*$, and the solutions $\tilde{s}$ and $\tilde{s}^*$; this yields an isomorphism $\tilde{\chi} : \cl(A \cup \{b\}) \rightarrow \cl(A^* \cup \{b^*\})$ of $L_n$-structures. The fact that $\tilde{\chi}$ extends $\chi$ follows from the uniqueness clause of Proposition~\ref{isomorphisms_stars} applied to $\tilde{\chi} \upharpoonright \cl(A)$ and $\chi$.

Given Claim~\ref{back_and_forth}, we now consider any two connected $n$-ary polygroupoids with inverses $\mathcal{H} = (I, P_2, \ldots, G)$ and $\mathcal{H}^* = (I^*, P^*_2, \ldots, G^*)$ such that $|I| = |I^*| \geq n$ and $G \cong G^*$. First note that for any $n$-element sets $A \subseteq I$ and $A^* \subseteq I^*$, there is an isomorphism $\chi: \cl(A) \rightarrow \cl(A^*)$ of $L_n$-structures which extends $\psi: G \rightarrow G^*$ (this is straightforward, but since $\cl(A)$ includes fibers $P_n(u)$ for all $n!$ permutations of $A$, we do need properties~(6) and (7) of Propostion~\ref{std_model_Tng} to ensure the existence of $\chi$). With this as a base case, we can now use Claim~\ref{back_and_forth} as the induction step of a standard back-and-forth argument to conclude that $\mathcal{H} \cong \mathcal{H}^*$ as $L_n$-structures. This gives us the completeness of the axioms in part~(1) of Theorem~\ref{total_cat_qe} plus the total categoricity of $T_{G,n}$. Finally, the elimination of quantifiers for $T_{G,n}$ in the language $L_n$ plus constants for $G$ also follows immediately from Claim~\ref{back_and_forth} by a standard argument.

\end{proof}

Everywhere below, $G$ is always a finite nontrivial abelian group.

\begin{lemma}
\label{independence}
If $\mathfrak{C}$ a monster model of $T_{G,n}$ and $A, B, C$ substets of $I(\mathfrak{C}) \cup P(\mathfrak{C})$, $$A \ind_B C \Leftrightarrow \left[ \cl(A \cup B) \cap \cl(C \cup B) \right] = \cl(B).$$
\end{lemma}

\begin{proof}
The direction ``$\Rightarrow$'' is immediate using the characterization of algebraic closure given by Lemma~\ref{acl_homesort}. For the direction ``$\Leftarrow$,'' suppose that $A,B$, and $C$ are subsets of $\mathfrak{C}$. By quantifier elimination, the only formulas $\varphi(\overline{x}; \overline{c})$ over $C$ which divide over $B$ are ones which imply that $\overline{x} \in \cl(C \cup B) \setminus \cl(B)$ (since the atomic $L_n$-formulas over $\overline{c}$ can only express that a variable from $\overline{x}$ is in some finite substructure generated by $\overline{c}$). But since $T_{G,n}$ is stable, forking equals dividing. Therefore if $\left[ \cl(A \cup B) \cap \cl(C \cup B) \right] = \cl(B),$ then $A \ind_B C$.
 \end{proof}



\begin{proposition}
\label{weak_ei}
After naming constants for every element of $G$, $T_{G,n}$ has weak elimination of imaginaries (that is, elimination of imaginaries up to names for finite sets).
\end{proposition}

\begin{proof}
We use the criterion given by Lemma~16.17 of \cite{poizat}, which says that it suffices to check the following two conditions (working in $\mathfrak{C}$ and not $\mathfrak{C}^{eq}$):

\begin{enumerate}
\item There is no strictly decreasing sequence $A_0 \supsetneq A_1 \supsetneq \ldots$, where every $A_i$ is the algebraic closure of a finite set of parameters; and
\item If $A$ and $B$ are algebraic closures of finite sets of parameters in $\mathfrak{C}$, then $\Aut(\mathfrak{C} / A \cap B)$ is generated by $\Aut(\mathfrak{C} / A)$ and $\Aut(\mathfrak{C} / B)$.
\end{enumerate}

Condition (1) is immediate from the characterization of algebraic closure given by Lemma~\ref{acl_homesort} (that each $A_i$ is $\cl(I_i)$ for some $I_i \subseteq I(\mathfrak{C})$). 

For condition (2), suppose that $\sigma \in \Aut(\C / A \cap B)$, and assume that $A = \cl(A_0)$ and $B = \cl(B_0)$ where $A_0, B_0 \subseteq I(\C)$.  By Proposition~\ref{isomorphisms_stars}, any permutation of $I(\C)$ which fixes $A_0$ can be extended to an automorphism of $\Aut(\C / A)$, and likewise for $B_0$ and $B$.  

Using the fact that $\textup{Sym}(I/ A_0 \cap B_0)$ is generated by $\textup{Sym}(I/ A_0)$ and $\textup{Sym}(I/ B_0)$ to find an automorphism $\tau \in \Aut(\C)$ such that $\tau$ is in the subgroup generated by $\Aut(\C / A)$ and $\Aut(\C / B)$ and $\widetilde{\sigma}:=\sigma \circ \tau^{-1}$ fixes $I$ pointwise.

Now we write $\widetilde{\sigma}$ as the composition of two automorphisms $\sigma_A\in \Aut(\C/A)$ and $\sigma_B\in \Aut(\C/B)$. Take an arbitrary point $a\in A_0\cap B_0$ (if the intersection is empty, choose $a\in A_0$). Let $\mathcal{S}$ be an $n$-star at $a$ and let $s:\mathcal{S}\to P_n$ be a solution function. Let $s^*=\widetilde{\sigma}\circ s$. Then $s^*$ is a solution function on the same $n$-star $\mathcal{S}$. Let $t:\mathcal{S}\to G$ be the function uniquely determined by the condition $s^*(a,c_1,\dots,c_{n-1})=s(a,c_1,\dots,c_{n-1})+t(a,c_1,\dots,c_{n-1})$.

By Proposition~\ref{isomorphisms_stars}, $\widetilde{\sigma}$ is completely determined by $s$ and $s^*$ (because it fixes $I$ and $G$). Equivalently, the isomorphism $\widetilde{\sigma}$ is uniquely determined by the function $t$. We now show how to write $t$ as a sum of two functions $t_A$ and $t_B$ that give isomorphisms fixing $A$ and $B$ respectively.

Since $\widetilde{\sigma}$ fixes $A\cap B$ pointwise, we have 
$$
t(a,c_1,\dots,c_{n-1})=0\textrm{ for all } c_1,\dots,c_{n-1}\in A_0\cap B_0. \qquad (*)
$$
Define the functions 
\begin{gather*}
t_A(a,c_1,\dots,c_{n-1})=\begin{cases} 0, & \textrm{if } c_1,\dots,c_{n-1}\in A_0;\\
t(a,c_1,\dots,c_{n-1}), & \textrm{otherwise}; \end{cases}
\\
t_B(a,c_1,\dots,c_{n-1})=\begin{cases} t(a,c_1,\dots,c_{n-1}), & \textrm{if } c_1,\dots,c_{n-1}\in A_0;\\
0, & \textrm{otherwise}; \end{cases}.
\end{gather*}
It is clear that $t=t_A+t_B$. For all $c_1,\dots,c_{n-1}\in B_0$ we have $t_B(a,c_1,\dots,c_{n-1})=0$: if at least one of $c_1,\dots,c_{n-1}$ is not in $A_0$, this follows by the definition; and if $c_1,\dots,c_{n-1}\in A_0\cap B_0$, the function $t_B(a,c_1,\dots,c_{n-1})=t(a,c_1,\dots,c_{n-1})=0$ by $(*)$.

Let $\sigma_A \in \Aut(\C)$ be the automorphism determined by the identity permutation of $I$, identity isomorphism of $G$ and the function $t_A$. Similarly let $\sigma_B \in \Aut(\C)$ be the automorphism determined by the functon $t_B$. Since the function $t_A(a,c_1,\dots,c_{n-1})=0$ for all $c_1,\dots,c_{n-1}\in A_0$ and $t_B(a,c_1,\dots,c_{n-1})=0$ for all $c_1,\dots,c_{n-1}\in B_0$, we have $\sigma_A \in \Aut(\C/A)$ and $\sigma_B \in \Aut(\C/B)$. Finally, since $t_A(a,c_1,\dots,c_{n-1})+t_B(a,c_1,\dots,c_{n-1})=t(a,c_1,\dots,c_{n-1})$ for all $c_1,\dots,c_{n-1}\in I$, we have $\sigma_B\circ \sigma_A =\widetilde{\sigma}$. Finally, we get $\sigma_B \circ \sigma_A \circ \tau = \sigma$.
\end{proof}

\begin{proposition}
 \label{standard_amalg}
 \begin{enumerate}
\item $T_{G,n}$ has $(\leq n)$-uniqueness.
\item $T$ does not have $(n+1)$-uniqueness over $\acl(\emptyset)$.
\item If $A \subseteq \mathfrak{C}$ is any algebraically closed set containing at least one element of $I$, then $T$ has $k$-uniqueness over $A$ for every $k$.
\end{enumerate}
 \end{proposition}
 
\begin{proof}
(1) We will show that $T_{G,n}$ has $n$-uniqueness over any set, and note that $k$-uniqueness for $k < n$ can be proved by a similar argument. Let $\cA: \mathcal{P}^-([n]) \rightarrow \cC$ be a closed independent $n$-amalgamation problem over the algebraically closed set $B \subseteq \mathfrak{C}$ (as in the notation of Section~1), and let $\cA', \cA''$ be two independent solutions to $\cA$. If $\cA(\{i\})$ is the algebraic closure over $B$ in $T_{G,n}^{eq}$ of $m_i$ elements from $I \setminus B$ (for each $i \in [n]$), then the image of $\cA'_{\{i\}, [n]}$ is the algebraic closure over $B$ of some tuple $\overline{b}_i \subseteq I$ of size $m_i$, and the image of $\cA''_{\{i\}, [n]}$ is the algebraic closure over $B$ of a tuple $\overline{c}_i \subseteq I$ of size $m_i$, where both families $\{\overline{b}_i : i \leq n \}$ and $\{\overline{c}_i : i \leq n\}$ are pairwise disjoint. Pick some $a \in I$ which is not contained in $B$ nor in any of the tuples $\overline{b}_i$ or $\overline{c}_i$. Choose some $n$-star $\mathcal{S}$ centered at $a$ over the set $B \cup \{a\} \cup \{\overline{b}_i : i \leq n\}$ and another $n$-star $\mathcal{S}^*$ centered at $a$ over $\supp(B) \cup \{a\} \cup \{\overline{c}_i : i \leq n\}$. We arbitrarily select a solution function $s: \mathcal{S} \rightarrow P_n$ as in Proposition~\ref{isomorphisms_stars}, and note that for any $u \in \mathcal{S}$, the element $s(u)$ must be completely contained in the image of some transition map $\cA'_{v, [n]}$ for $v$ a proper subset of $[n]$. Fix some bijection $$\varphi : \supp(B) \cup \{a\} \cup \{\overline{b}_i : i \leq n \} \rightarrow \supp(B) \cup \{a\} \cup \{\overline{c}_i : i \leq n\}$$ which fixes $\supp(B) \cup \{a\}$ pointwise. Then $\mathcal{S}$ is transferred via $\varphi$ to some $n$-star $\mathcal{S}^*$ centered at $a$ over the set $\supp(B) \cup \{a\} \cup \{\overline{c}_i : i \leq n \}$, and there is a unique solution function $s^* : \mathcal{S}^* \rightarrow P_n$ such that for any $u \in \mathcal{S}$, there is some proper subset $v$ of $[n]$ and some $f \in P_n$ such that $s(u) = \cA'_{v,[n]}(f)$ and $s^*(\varphi(u)) = \cA''_{v, [n]}(f)$. By Proposition~\ref{isomorphisms_stars}, we can construct a unique isomorphism $$\chi : \cl(B \cup \{a\} \cup \{\overline{b}_i : i \leq n\}) \rightarrow \cl(B \cup \{a\} \cup \{\overline{c}_i : i \leq n \})$$ extending $\varphi$ and the identity map on $G$ and sending $s$ to $s^*$, and $\chi$ is an elementary map since $T_{G,n}$ eliminates quantifiers. This gives the desired isomorphism between $\cA'$ and $\cA''$.


(2) Let $(a_1, \ldots, a_{n+1})$ be a sequence of distinct elements of $I(\mathfrak{C})$, which form an independent set (by Lemma~\ref{independence}), then pick any elements $f_1, \ldots, f_{n+1}$ such that $f_i \in P(a_1, \ldots, \widehat{f_i}, \ldots, f_{n+1})$ and $Q(f_1, \ldots, f_{n+1})$ holds, and clearly $f_i \in \acl(\overline{a}_{\neq i}) \setminus \dcl(\overline{a}_{\neq i})$, so $T$ does not satisfy the property $B(n+1)$ over $\acl(\emptyset)$.

(3) To show $k$-uniqueness over $A$ with $a \in A \cap I$, consider two solutions $\cA'$ and $\cA''$ to the same $k$-amalgamation problem $\cA$, and we repeat essentially the same argument as in the proof of (1), only this time we use $n$-stars $\mathcal{S}$ and $\mathcal{S}^*$ over the given point $a$. Then $$\cA'([k]) = \cl (\{s(u) : u \in \mathcal{S} \})$$ and any $s(u)$ for $u \in \mathcal{S}$ is automatically in the image of some transition map $\cA'_{v, [k]}$ for some $v \subset_{k-1} [k]$, so by Proposition~\ref{isomorphisms_stars} we can build the desired isomorphism from $\cA'([k])$ to $\cA''([k])$.

\end{proof}

This concludes our discussion of models of $T_{G,n}$, which establishes that there do indeed exist connected $n$-ary polygroupoids for every $n \geq 2$ definable in theories that are very well-behaved ($\omega$-categorical and $\omega$-stable), and that for any finite nontrivial abelian $G$ we can find such a structure on which $G$ acts regularly (in the sense of Definition~\ref{group_action} above). It seems that it is possible to construct more complicated examples of $n$-ary polygroupoids, for example in which the sorts $P_k$ for $k \in \{2, \ldots, n-1\}$ have more than one fiber above every tuple in $I^{(k)}$, or even with more complicated ``inverse maps'' $\iota_\sigma : P(a_1, \ldots, a_n) \rightarrow P(a_{\sigma(1)}, \ldots, P_{\sigma(n)})$ such that $\iota^2_{(12)} \neq \id$ for the transposition $(12)$. We will leave a detailed study of all these possibilities for a future paper.

\end{document}